\documentclass[12]{amsart}

\usepackage{amssymb}
\usepackage{amsmath}
\usepackage{graphicx}
\usepackage{eucal}
\usepackage{epstopdf}
\usepackage{mathrsfs}
\usepackage{enumerate}
\usepackage{fullpage} 
\usepackage{setspace}
\usepackage{color}
\usepackage{dsfont}

\begin{document}

\onehalfspacing

\newtheorem{theorem}{Theorem}[section]
\newtheorem{lemma}[theorem]{Lemma}
\newtheorem{proposition}[theorem]{Proposition}
\newtheorem{corollary}[theorem]{Corollary}
\newtheorem{prop1}[theorem]{Proposition (Rankin)}
\newtheorem{prop2}[theorem]{Proposition (Rankin and Selberg)}
\newtheorem{prop3}[theorem]{Proposition (Rudnick and Sarnak)}
\newtheorem{conj}[theorem]{Conjecture}

\theoremstyle{definition}
\newtheorem{definition}[theorem]{Definition}
\newtheorem{example}[theorem]{Example}
\newtheorem{xca}[theorem]{Exercise}

\theoremstyle{remark}
\newtheorem{remark}[theorem]{Remark}

\numberwithin{equation}{section}

\newcommand{\abs}[1]{\lvert#1\rvert}
\newcommand{\kk}{\kappa}
\newcommand{\al}{\alpha}

\title[Bounds for Additive Divisor Sums]{Bounds and Conjectures for Additive Divisor Sums}

\author{Nathan Ng}
\address{Department of Mathematics and
  Computer Science \\
  University of Lethbridge \\
  Lethbridge, AB \\
  Canada T1K 3M4}
\email{nathan.ng@uleth.ca}

\dedicatory{In memory of Kevin Henriot.}

\author{Mark Thom}
\email{mark.thom2@uleth.ca}


\date{\today}

\keywords{divisor functions, additive divisor sums}

\begin{abstract}
Additive divisor sums play a prominent role in the theory of the
moments of the Riemann zeta function.  There is a long history of
determining sharp asymptotic formula for the shifted convolution sum
of the ordinary divisor function.  In recent years, it has emerged
that a sharp asymptotic formula for the shifted convolution sum of the
triple divisor function would be useful in evaluating the sixth moment
of the Riemann zeta function.  In this article, we study $D_{k,\ell}(x) = \sum_{n \le x} \tau_k(n) \tau_{\ell}(n+h)$
where $\tau_k$ and $\tau_{\ell}$ are the $k$-th and $\ell$-th divisor functions. 
The main result is a lower bound
of the correct order of magnitude for $D_{k,\ell}(x,h)$, uniform in $h$.  In addition, the
conjectural asymptotic formula for $D_{k,\ell}(x,h)$ is
studied.   Using an argument of Ivi\'{c} \cite{Iv1}, \cite{Iv2} and Conrey-Gonek \cite{CG} the 
leading term in the conjectural asymptotic formula  is simplified.
In addition,  a probabilistic method is presented which gives the same leading term.
Finally, we show that these two methods give the same answer as in a recent probabilistic argument of Terry Tao
\cite{Tao}. 
\end{abstract}

\maketitle

\section{Introduction and main theorem} \label{Intro}

Many important problems in analytic number theory concern sums of the form 
\begin{equation}
   \label{fgshifted}
   \sum_{n \le x} f(n) g(n+h)
\end{equation}
where $h \in \mathbb{N}$ and $f$ and $g$ are arithmetic functions.  For instance, the twin prime conjecture
would follow from an asymptotic evaluation of  \eqref{fgshifted} with $f=g=\Lambda$, the von Mangoldt function. 
If $f=g=\lambda$, the Liouville function, this is a special case of the sum that occurs in Chowla's conjecture. 
In this article, we focus on \eqref{fgshifted} with $f=\tau_k$ and $g=\tau_{\ell}$, the  $k$-th and $\ell$-th divisor functions
where $k,\ell \in \mathbb{N}$.
For $n \in \mathbb{N}$, the $k$-th divisor function is defined by 
\[
\tau_k(n) = \# \{ (n_1, \ldots, n_k) \in \mathbb{N}^k \ | \ n_1 \cdots n_k = n \}. 
\] 
Equivalently, $\tau_k(n)$ is the coefficient of $n^{-s}$ in the Dirichlet series of $\zeta(s)^k$, 
where $\zeta(s)$ is the Riemann zeta function. 
Our main focus is the correlation sum
\begin{equation}
  \label{Dklxh}
D_{k,\ell}(x, h) := \sum_{n \le x} \tau_k(n)\tau_{\ell}(n+h) \text{ with } h \in \mathbb{N}.
\end{equation}
For $k=\ell$, we shall use the abbreviated notation
\begin{equation}
  \label{Dkxh}
D_k(x,h) := D_{k,k}(x, h) := \sum_{n \le x} \tau_k(n)\tau_{k}(n+h) \text{ with } h \in \mathbb{N}.
\end{equation}
This last sum has been extensively studied. 
For $k=1$, this sum is trivial. For $k=2$, there is a rich theory connecting this sum to the spectral theory of automorphic forms.   
However,  for $k>2$, this sum is mysterious and there are few results.  Nevertheless, there is the following conjecture: 

\begin{conj} (Additive Divisor Conjecture:  simplified version) \\
Let $\varepsilon >0$ and $k, \ell \ge 2$.  For $1 \le h \le x^{1-\varepsilon}$, we have 
\begin{equation}
   \label{Dklasymptotic}
    D_{k,\ell}(x, h)  \sim \frac{ c_{k,\ell}(h)}{(k-1)! (\ell-1)!} x (\log x)^{k+\ell-2}
\end{equation}
as $x \to \infty$, for a certain real valued constant $c_{k,\ell}(h)$ given by \eqref{cklNgThom} and \eqref{cklTao} below.
\end{conj} 
In this article we provide several expressions for $c_{k,\ell}(h)$.    The value for $c_{k,\ell}(h)$ 
can be computed using the work of Ivi\'{c} \cite{Iv1} and of Conrey-Gonek \cite{CG}.  Both of these papers use the $\delta$-method (circle method) to give a formula for $D_{k}(x,h)$. 
In addition, we present a heuristic probabilistic method in section \ref{Probmethod} to give an alternate calculation 
of $c_{k,\ell}(h)$.  These two methods lead to
\begin{equation}
  \label{cklNgThom}
  c_{k,\ell}(h) = C_{k,\ell} f_{k,\ell}(h),
\end{equation}
where 
\begin{equation}
  \label{Ckl}
  C_{k,\ell} :=  \prod_p \Big(  \Big(1-\frac{1}{p} \Big )^{k-1}
  +   \Big(1-\frac{1}{p} \Big )^{\ell-1}
   -
       \Big ( 1-\frac{1}{p} \Big )^{k+\ell-2} \Big ),
\end{equation}
and $f_{k,\ell}(\cdot)$ is a multiplicative function  defined  on prime powers $p^{\alpha}$ by
\begin{equation}
  \label{fklpa}
 f_{k,\ell}(p^{\alpha}) :=\frac{
 1 + \sum_{i=1}^{\alpha} (\tau_k(p^i)\tau_{\ell}(p^i)-\tau_k(p^{i-1})\tau_{\ell}(p^{i-1}))X^i 
    +    \sum_{i=\alpha+1}^{\infty} ( \tau_k(p^{\alpha})  \tau_{\ell-1}(p^i)
   + \tau_{\ell}(p^{\alpha})   \tau_{k-1}(p^i))X^i
 }{  \Big( 1-\frac{1}{p} \Big)^{-(k-1)}+\Big( 1-\frac{1}{p} \Big)^{-(\ell-1)}-1  }.
\end{equation}
We also provide several other expressions for $f_{k,\ell}(p^{\alpha})$ and hence $c_{k,\ell}(h)$ 
(see \eqref{fklpa2}, \eqref{fklpa3}, and \eqref{cklhexpectedvalue} below).
 Another expression for 
$c_{k,\ell}(h)$ has been given by Terry Tao. 
In a blogpost of Aug. 31, 2016, Tao provided a different heuristic probabilistic argument that gives
\begin{equation}
   \label{cklTao}
   c_{k,\ell}(h) = \prod_{p}  \mathfrak{S}_{k,\ell,h}(p)
\end{equation}
where 
\begin{equation}
   \label{Sklhp}
   \mathfrak{S}_{k,\ell,h}(p) = \Big( 1- \frac{1}{p} \Big)^{k+\ell-2} \sum_{j \ge 0: p^j \mid h}  \frac{1}{p^j} 
   P_{k,\ell,p}(j),
\end{equation}
\begin{equation}
  \label{Pklpj}
  P_{k,\ell,p}(j) = \sum_{k'=2}^{k} \sum_{\ell'=2}^{\ell} 
   \binom{k-k'+j-1}{k-k'} \binom{\ell-\ell'+j-1}{\ell-\ell'}
  \Big( \Big( \frac{p}{p-1} \Big)^{k'-1} + \Big( \frac{p}{p-1} \Big)^{\ell'-1}-1
  \Big)
\end{equation}
and the conventions $\binom{-1}{0}=1$ and $\binom{m-1}{m}=0$ for $m \ge 1$ are used here. 
This expression can be further simplified to 
\begin{equation}
\begin{split} 
  \label{Pklpj2}
   & P_{k,\ell,p}(j) = \\
  &   
   \binom{k+j-2}{j}
    \sum_{i=0}^{\ell-2} \binom{i+j-1}{i}   \Big( \frac{p}{p-1} \Big)^{\ell-i-1}
   + 
    \binom{\ell+j-2}{j}
    \sum_{i=0}^{k-2} \binom{i+j-1}{i}   \Big( \frac{p}{p-1} \Big)^{k-i-1}
    - \binom{k+j-2}{j} \binom{\ell+j-2}{j}.
\end{split}
\end{equation}
Although it is not obvious, we shall show in section \ref{Probmethod} that the expresssions for $c_{k,\ell}(h)$ given by \eqref{cklNgThom} and \eqref{cklTao}
are equal. 
It is not clear what is the simplest or most natural form for $c_{k,\ell}(h)$. Currently, 
\eqref{cklTao} with \eqref{Sklhp} and \eqref{Pklpj2} appears to be the simplest known expression for $c_{k,\ell}(h)$. 
%

The above conjecture simplifies conjectures of  Ivi\'{c} \cite{Iv1} and Conrey-Gonek \cite{CG},
though in the above formulation we allow $h$ to be as large as $x^{1-\varepsilon}$ instead of $x^{\frac{1}{2}}$. 
The case $h=1$ reduces to  
\[
  \sum_{n \le x} \tau_{k}(n) \tau_{\ell}(n+1) \sim \frac{C_{k,\ell}}{(k-1)! (\ell-1)!} x (\log x)^{k+\ell-2}. 
\]
The conjectures of \cite{Iv1} and \cite{CG} may be written in the form 
\begin{equation}
  \label{Dkxhasymptotic}
D_{k,\ell}(x,h) = x\Big( \alpha_0(h) (\log x)^{k+\ell-2} +   \sum_{i=1}^{k+\ell-2} \alpha_i(h)  (\log x)^{k+\ell-2-i} \Big) + o(x)
\end{equation}
for certain coefficients $\alpha_i(h)$ where $h$ is allowed to vary with $x$.  Ivi\'{c} \cite{Iv1} gave formulae for the $\alpha_i(h)$
in terms of certain singular series.   On the other hand, Conrey and Gonek gave a formula for the derivative of the above
main term
in terms of a complicated double complex integral.
This will be discussed in further detail in section two where we show that $\alpha_0(h) = C_{k,\ell} f_{k,\ell}(h)/(k-1)!(\ell-1)!$.

The main result in this article is a uniform lower bound for $D_{k,\ell}(x,h)$. 
\begin{theorem} \label{mainthm}
For $k,\ell \ge 3$, there exists $B_{k,\ell}>0$ such that for $h \le \exp(B_{k,\ell} (\log x \log \log x)^{\frac{\min(k,\ell)-1}{\min(k,\ell)-1.99}})$, 
we have 
\[
   \frac{1}{2^{k+\ell-2}} \frac{C_{k,\ell} f_{k,\ell}(h)}{(k-1)! (\ell-1)!}x (\log x)^{k+\ell-2}
   \Big( 1+ O_{k,\ell} \Big( \frac{\log \log h}{\log x}  \Big) \Big)
    \le
   D_{k,\ell}(x,h)
\]
as $x \to \infty$. 
\end{theorem}
Recently, Kevin Henriot  informed us that S. Daniel \cite{Da} showed that 
\begin{equation}
  \label{Danbd}
  D_{k,\ell}(x,h) \ll_k  
  \prod_{p \mid h} \Big(1+ \frac{(k-1)(\ell-1)}{p} \Big) x (\log x)^{k+\ell-2}, \text{ for } h \le x^{C},
\end{equation}
for any $C >0$.  Note that since 
\begin{equation}
  \label{fkpaestimate}
  f_{k,\ell}(p^{\alpha}) = 1 + \frac{(k-1)(\ell-1)}{p} + O_{k,\ell}(p^{-2})
\end{equation}
\eqref{Danbd} implies 
\begin{equation}
  \label{Danbd2}
  D_{k,\ell}(x,h) \ll_{k,\ell}  f_{k,\ell}(h)  
 x (\log x)^{k+\ell-2},  \text{ for } h \le x^{C},
\end{equation}
for any $C >0$.  Unfortunately, this result was never published. 
However, Henriot has shown us a proof \cite{He2}
based on \cite{He} and \cite{HeEr}.   In \cite{He} he establishes bounds for 
\begin{equation}
  \label{nairten}
  \sum_{x < n \le x+y} \tau_{k_1}(|Q_1(n)|) \tau_{k_2}(|Q_2(n)|)  \cdots \tau_{k_J}( |Q_J(n)|)
\end{equation}
where $Q_j$ are polynomials with integer coefficients.  More generally he bounds 
\begin{equation}
    \label{nairten2}
  \sum_{x < n \le x+y} f_{1}(|Q_1(n)|) f_{2}(|Q_2(n)|)  \cdots f_{J}( |Q_J(n)|)
\end{equation}
where the $f_i$ belong to a general class of multiplicative functions. 
Such expressions were originally considered by Nair and Tenenbaum \cite{NT}.
However, their bounds for \eqref{nairten2} were not uniform in the coefficients of the $Q_j$.
This problem was addressed by Daniel \cite{Da} and Henriot \cite{He}.  Recently Klurman \cite{K} has 
obtained some interesting results for \eqref{nairten2} in the case that the images of the multiplicative functions
$f_i$ lie in the unit disc.  
Theorem \ref{mainthm} and \eqref{Danbd2} lead us to propose the following  problem. \\
\noindent {\bf Problem}.  Let $k,\ell \ge 3$. Determine the best explicit constants $c_1=c_1(k,\ell)$ and $c_2=c_2(k,\ell)$ 
such that 
\[
     c_1    \le 
     \frac{D_{k,\ell}(x,h)}{ \frac{c_{k,\ell}(h) }{(k-1)! (\ell-1)!} x (\log x)^{2k-2}} \le c_2,
\]
uniformly for $h \le x^{1-\varepsilon}$, as $x \to \infty$.  \\ 
Theorem 1 yields $c_1 = \frac{1}{2^{k+\ell-2}} -\varepsilon$.   
and \eqref{Danbd2} yields
$c_2 = O_{k,\ell}(1)$. Henriot has suggested that  in the case $k=\ell$ the proof of \eqref{Danbd} demonstrates that $c_2$ is doubly or triply exponential in $k$. 

To finish this section, we give some properties of divisor functions, list our conventions and notation, 
and provide an outline of the article. 

\subsection{Properties of Divisor functions} \label{divisorproperties}

This article makes extensive use of divisor functions and related arithmetic functions.  
Recall that for $k \in \mathbb{N}$, the $k$-th divisor function satisfies 
\begin{equation}
  \label{divseries}
  \sum_{j=0}^{\infty} \tau_{k}(p^j) X^{j} = (1-X)^{-k}
\end{equation}
for $p$ prime and $|X| < 1$.  It follows that for $p$ prime and $j \ge 0$,  
\begin{equation}
  \label{taukpj}
  \tau_{k}(p^j) =  \binom{k+j-1}{j}.
\end{equation}
The divisor functions satisfy the relation
\begin{equation}
  \label{recurrence}
   \tau_{k-1}(p^{j}) =\tau_{k}(p^j)-\tau_{k}(p^{j-1}) \text{ for } p \text{ prime}, k,  j \ge 1. 
\end{equation}
We shall also encounter a multiplicative function $\sigma_{k}(\cdot,s):\mathbb{N} \to \mathbb{C}$,
where $k \in \mathbb{N}$, $s \in \mathbb{C}$.   
For $n \in \mathbb{N}$,  it is defined by 
\begin{equation}
   \label{sigmams}
   \sigma_{k}(n,s)= \Big( \sum_{a=1}^{\infty} \frac{\tau_{k}(na)}{a^s}  \Big) \zeta(s)^{-k}.
\end{equation}
By multiplicativity, it follows that  
\begin{equation}
   \label{sigmapjid}
   \sigma_{k}(p^j,s)  
    = \frac{\sum_{i=0}^{\infty} \frac{\tau_{k}(p^{j+i})}{p^{is}}}{ \sum_{i=0}^{\infty} \frac{\tau_{k}(p^i)}{p^{is}}}
    = (1-p^{-s})^{k} \sum_{i=0}^{\infty} \frac{\tau_{k}(p^{j+i})}{p^{is}}  
\end{equation}
for $j \ge 1$, and in particular,
\begin{equation}
  \label{sigmaps}
   \frac{\sigma_{k}(p,s)}{p^s} = 1- (  1- p^{-s} )^{k}. 
\end{equation}
Moreover, it was proven in \cite{Ng} that 
\begin{equation}
  \label{ngid}
  \sigma_{k}(p^j,s) = \tau_{k}(p^j) H_{k,j}(p^{-s})
\end{equation}
where 
\begin{equation}
    \label{Hkj}
   H_{k,j}(x) 
   :=  j x^{-j} \int_{0}^{x} t^{j-1}(1-t)^{k-1} \, dt   \hspace{0.3cm} \footnote{In \cite{Ng}, we used the notation $H_{j,k}(x)$
instead of $H_{k,j}(x)$.}
\end{equation}
for $j, k \in \mathbb{N}$.  Repeated integration by parts of \eqref{Hkj}
leads to the representation 
\begin{equation}
  \label{Hkj2}
   H_{k,j}(x) :=   
    \sum_{i=0}^{k-1}
            \frac{{k-1 \choose i}}{{j+i \choose j}}
            ( 1-x)^{k-1-i} 
           x^i \text{ where } k \in \mathbb{N}, j \in \mathbb{Z}_{\ge 0}.
\end{equation}
Note that $H_{k,j}(x)$ is a degree $k-1$ polynomial and $H_{k,j}(0)=1$.
Later in the article, we show that 
\begin{equation}
  \label{fklpa2}
   f_{k,\ell}(p^{\alpha}) =  \frac{
    \sum_{j=0}^{\alpha} 
    \Big( \frac{\sigma_{k-1}(p^j,1)\sigma_{\ell-1}(p^j,1) }
            {p^{j}} - \frac{\sigma_{k-1}(p^{j+1},1)\sigma_{\ell-1}(p^{j+1},1) }{p^{j+2}} \Big)}{
       \Big( 1-\frac{1}{p} \Big)^{k-1} + \Big( 1-\frac{1}{p} \Big)^{\ell-1}  - 
       \Big ( 1-\frac{1}{p} \Big )^{k+\ell-2}}.
\end{equation}
By \eqref{Hkj2} we also have 
\begin{equation}
  \label{fklpa3}
   f_{k,\ell}(p^{\alpha}) = 
   \frac{
   \sum_{j=0}^{\alpha} 
    \Big( \frac{\tau_{k-1}(p^j) \tau_{\ell-1}(p^j)   H_{k-1,j}(p^{-1})H_{\ell-1,j}(p^{-1}) }
            {p^{j}} - \frac{\tau_{k-1}(p^{j+1})\tau_{\ell-1}(p^{j+1})  H_{k-1,j+1}(p^{-1})   H_{\ell-1,j+1}(p^{-1})}{p^{j+2}} \Big)}
  { \Big( 1-\frac{1}{p} \Big)^{k-1} + \Big( 1-\frac{1}{p} \Big)^{\ell-1}  - 
       \Big ( 1-\frac{1}{p} \Big )^{k+\ell-2}
       }.
\end{equation}
At several points in this article we make use of these representations. 

\subsection{ Conventions and notation}
In this article we shall use the convention that $\varepsilon$ denotes an arbitrarily small positive constant which may vary from line to line. Given two functions $f(x)$ and $g(x)$, we shall interchangeably use the notation  $f(x)=O(g(x))$, $f(x) \ll g(x)$, and $g(x) \gg f(x)$  to mean there exists $M >0$ such that $|f(x)| \le M |g(x)|$ for  sufficiently large $x$. 
If we write $f(x)=O_{k,\ell}(g(x))$, $f(x) \ll_{k,\ell} g(x)$, or $f(x) \asymp_{k,\ell} g(x)$, then we mean that the 
corresponding constants depend on $k$ and $\ell$.  The letter $p$ will always be used to denote a prime number.
For a complex valued, differentiable function
 $F: \mathbb{C}^2 \to \mathbb{C}$ and $i_1,i_2 \in \mathbb{Z}_{\ge 0}$  we write 
\begin{equation}
  \label{partials}
  F^{(i_1,i_2)}(s_1,s_2) := \frac{\partial^{i_1}}{\partial s_{1}^{i_1}} \frac{\partial^{i_2}}{\partial s_{2}^{i_2}} 
  F(s_1,s_2)
\end{equation}
where $ \frac{\partial^{i}}{\partial s^{i}}$ denotes the $i$-th partial derivative with respect to $s$.  \\
Given $a,b \in \mathbb{Z}$, we let $(a,b)$ denote the greatest common divisor of $a$ and $b$
and $[a,b]$ denotes the least common multiple of $a$ and $b$.

\subsection{Organization of the article}
The article is organized as follows. In section \ref{History} the conjectural asymptotic formula for $D_{k,\ell}(x,h)$ is studied
based on the work of Ivi\'{c} \cite{Iv1} and Conrey-Gonek \cite{CG}.   
We show that the leading term in the asymptotic formula for $D_{k,\ell}(x,h)$ is $\frac{c_{k,\ell}(h)}{(k-1)! (\ell-1)!}x (\log x)^{k+\ell-2}$.   
In section \ref{Lowerbound}, the lower bound in Theorem \ref{mainthm} is proven. 
In section \ref{Probmethod},  a simple probabilistic method is used to rederive the main term of $D_{k,\ell}(x,h)$ which agrees with 
the calculation in section \ref{History}.
In addition, we show that our constant for $c_{k,\ell}(h)$ \eqref{cklNgThom} agrees with Tao's \eqref{cklTao}.
 Finally, we discuss open problems related to additive divisor sums
and avenues for future research.

\section{A brief history of additive divisor sums and a conjectural formula for $D_{k,\ell}(x,h)$} \label{History}

\subsection{A history of additive divisor sums}
Questions concerning sums of the form $D_{k,\ell}(x,h)$ are called additive divisor problems. 
These functions are of interest due to the well-known connection between $D_k(x,h)$ and the $2k$-th moments of the Riemann zeta function, defined 
by 
\[ 
I_k(T) = \int_{0}^{T} |\zeta(\tfrac{1}{2}+it)|^{2k} dt \text{ for } k \ge 0.
\]  
In 1926, Ingham \cite{In} discovered that $D_{2}(x,h)$ is intimately related to the fourth moment, $I_2(T)$. 
He succeeded in proving that 
\[
I_2(T) \sim \frac{T}{2 \pi^2}  (\log T)^4
\]
and an important part of his argument made use of the inequality
\[
  D_{2}(x,h) \ll   \sigma_{-1}(h) x (\log x)^2
\]
for $h \le x$, 
where $\sigma_{-1}(h) =  \sum_{d \mid h} d^{-1}$. 
In \cite{In2} he improved this to 
\begin{equation}
  \label{Ingham}
   D_{2}(x,h) \sim \frac{6}{\pi^2}  \sigma_{-1}(h)  x\log^2{x}. 
\end{equation}
In 1931,    Estermann \cite{Es} proved an estimate of the shape 
\begin{equation}
  \label{Estermann}
D_{2}(x,h) =  x \Big( \frac{6}{\pi^2} \sigma_{-1}(h) \log^2 x + \alpha_1(h) \log x + \alpha_2(h)\Big) + O(x^{\theta+\varepsilon})
\end{equation}
with $\theta= \frac{11}{12}$ and  $\alpha_1(h)$ and $\alpha_2(h)$ are certain arithmetic functions.
Estermann's work relates $D_2(x,h)$ to a formula involving special exponential sums known as Kloosterman sums.
For $q$ a natural number and $u,v$ integers, the Kloosterman sum $S(u,v;q)$ is defined by 
$$
S(u,v;q) := \sum_{\substack{a = 1 \\ (a,q) = 1 \\ a\bar{a} \equiv 1 (\text{mod } q)}}^q e \Big(  \frac{ua + v\bar{a}}{q} \Big).
$$
These sums exhibit considerable cancellation and they arise  in many contexts in analytic number
theory. 
Estermann derived the non-trivial bound  $S(u,v;q) \ll q^{\frac{3}{4} + \varepsilon}(u,q)^{\frac{1}{4}}$ and this led to the error term in \eqref{Estermann}.   A famous result due to Weil is the bound:
$|S(u,v;q)| \le  \tau_2(q) (q,u,v)^{1/2} q^{1/2} \tau(q)$. 
Much later, Heath-Brown \cite{HB} made use of Weil's bound 
to obtain 
\eqref{Estermann} with $\theta=\frac{5}{6}$.  From this he deduced  that
there exists a degree four polynomial $Q_4$  such that 
\begin{equation}
  \label{I2T}
  I_2(T) = T Q_4(\log T) + O(T^{\Theta+\varepsilon}),
\end{equation}
where $\Theta =\frac{5}{6}$ is valid.  
The next advance was due to Deshouillers and Iwaniec \cite{DI}, who proved 
that \eqref{Estermann} is valid with $\theta=\frac{2}{3}$, in the case $h=1$.  
In their work, they related $D_{2}(x;1)$ to averages of Kloosterman sums and then 
made use of Kuznetsov's formula.    This  is a formula which relates sums of Kloosterman sums 
to the coefficients of Maass wave forms and holomorphic modular forms.  Motohashi extended this method and obtained \eqref{Estermann} with $\theta= \frac{2}{3}$, uniformly for  
$h \le x^{\frac{20}{27}}$.  He proved 
\begin{equation}
  \label{Motohashi}
D_{2}(x,h) = \frac{6}{\pi^2}\int_{0}^{\frac{x}{h}} q_2(t,h)\mbox{ }dt  + E_2(x,h)
\end{equation}
where 
\begin{equation}
\begin{split}
   \label{q2xh}
    q_2(t,h) & =
     \sigma(h) \log(t) \log(t+1) + ( \sigma(h)(2 \gamma-\frac{\zeta'}{\zeta}(2) -\log(h)) + 2\sigma^{(1)}(h)) \log(t(t+1)) \\
     & + \sigma(h) \Big( (2 \gamma-2 \frac{\zeta'}{\zeta}(2)-\log h)^2 -4 \Big( \frac{\zeta'}{\zeta}\Big)' (2)\Big)
     +4 \sigma^{(1)}(h) (2 \gamma-2 \frac{\zeta'}{\zeta}(2) - \log h) + 4 \sigma^{(2)}(h),
\end{split}
\end{equation}
$\sigma^{(j)}(h) := \sum_{d \mid h} d (\log d)^j$,
and $\gamma$ is Euler's constant and 
\begin{equation}
  E_2(x,h) =O((x(x+h))^{\frac{1}{3}+\varepsilon} + h^{\frac{9}{40}}(x(x+h))^{\frac{1}{4}+\varepsilon} + h^{\frac{7}{10}}x^{\varepsilon}).
\end{equation}
Related work of Motohashi establishes 
that $\Theta =\frac{2}{3}$ is valid in \eqref{I2T}.
Meurman \cite{Me} showed that 
\begin{equation}
  \label{Meurman}
    E_2(x,h) = O((x(x+h))^{\frac{1}{3}+\varepsilon} +(x(x+h))^{\frac{1}{4}} x^{\varepsilon} \text{min}(x^{\frac{1}{4}}, h^{\frac{1}{8}+\frac{\alpha}{2}})),
\end{equation}
where $\alpha$ is a positive constant which satisfies 
\begin{equation}
  \label{Ramanujan}
|\rho_j(n)| \le n^{\alpha} |\rho_j(1)|
\end{equation}
 where $\{ \rho_j(n) \}_{n=1}^{\infty}$
are the Fourier coefficients of an orthonormal basis of the space of non-holomorphic cusp forms
for the full modular group.

There are also results for $D_{k,\ell}(x,h)$.  
Linnik developed highly original techniques using ideas from additive number theory 
and probability theory, most notably the dispersion method \cite{LI} to deal with $D_{k,2}(x,h)$ with 
$k \ge 2$. He proved an asymptotic formula for $D_{k,2}(x,h)$, obtaining the leading term
with an error term.  The error term was improved by Motohashi \cite{Mo}, who used large sieve methods.  
Recently, Topacogullari \cite{To1} established a main term with a power savings in the case of 
$D_{3,2}(x,h)$.  This filled in details of results, stated without proof, by Deshouillers \cite{De} and Bykovski and Vinogradov \cite{BV}.
Furthermore, Drappeau \cite{Dr} has recently provided a main term with a power savings in the error term for $D_{k,2}(x,h)$ with $k \ge 3$ and this too has recently been improved by Topacogullari \cite{To3}.
Despite these impressive results,  no asymptotic formula for $D_{k,\ell}(x,h)$  has been proven 
in the case both $k$ and $\ell$ are greater than two.  We now present a conjectural formula for $D_{k,\ell}(x,h)$.
\subsection{A conjectural formula for $D_{k,\ell}(x,h)$}
We follow the work of Ivi\'{c} and Conrey and Gonek to work out the leading term of the conjectured
main term for $D_{k,\ell}(x,h)$.  We shall be concerned with an expression of the form 
$D_{k,\ell}(x,h) = m_{k,\ell}(x,h) + E_{k,\ell}(x,h)$
where $m_{k,\ell}(x,h)$ is the ``main term" and $E_{k,\ell}(x,h)$ is the ``error term." 
In \cite{Iv1}, \cite{Iv2}, and \cite{CG},  $m_{k,\ell}(x,h)$ was studied via Duke, Friedlander, and Iwaniec's \cite{DFI} version 
of the circle method, known as the 
$\delta$-method.   One of the key ideas of the circle method is to detect an additive condition via additive characters.
Consequently, it is important to have an asymptotic formula for the exponential sums $\sum_{n \le x} \tau_k(n) e(\frac{an}{q})$
where $(a,q)=1$ and $e(\theta)  := e^{2 \pi i \theta}$.  Naturally, one must understand the Dirichlet series 
$\sum_{n=1}^{\infty} \tau_k(n) e(\frac{an}{q}) n^{-s}$.    Ivi\'{c} \cite{Iv2} obtained a meromorphic continuation of 
this series by decomposing it in terms of Hurwitz zeta functions.  On the other hand, Conrey-Gonek  \cite{CG} obtained
a meromorphic continuation by expressing  $e(\frac{an}{q})$ in terms of multiplicative Dirichlet characters. They showed that 
\begin{equation}
  \label{taukadditive}
   \sum_{n \le x} \tau_k(n) e \Big(\frac{an}{q} \Big) \sim \frac{1}{q} \int_{0}^{x} P_k(t,q) dt
\end{equation}
where $P_k(t,q)$ is defined by 
\begin{equation} 
  \label{Pk}
   P_k(t,q) = \frac{1}{2 \pi i} \int_{C}  \zeta(s+1)^k  G_{k}(q,s+1) \Big( \frac{x}{q} \Big)^s ds,
\end{equation}
$C = \{ z \in \mathbb{C} \ | \ |z| = \eta\}$ for $0 < \eta < \frac{1}{10}$, and for $k \in \mathbb{N}$, $s \in \mathbb{C}$, $G_{k}(\cdot,s) : \mathbb{N} \to \mathbb{C}$ is the multiplicative function defined by 
\footnote{Conrey and Gonek use the notation $G_k(s,n)$, whereas we use $G_k(n,s)$.} 
\begin{equation}
  \label{Gksndef}
  G_{k}(n,s) =  \sum_{a \mid n} \frac{\mu(a) a^s}{\phi(a)} \sum_{b \mid a} \frac{\mu(b)}{b^s} \sigma_{k} \Big( \frac{nb}{a},s \Big). 
\end{equation}
Using \eqref{taukadditive}, the $\delta$-method leads to 
\begin{equation}
   \label{mkxhIvic}
   m_{k,\ell}(x,h)=  \int_{0}^{x} \sum_{q=1}^{\infty} \frac{c_q(h)}{q^2} P_k(t,q) P_{\ell}(t+h,q)  dt
\end{equation}
where $c_q(h)=\sum_{\substack{a=1 \\ (a,q)=1}}^{q} e(\tfrac{an}{q})$ is the Ramanujan sum.    
From the identity  $\log(t+h) = \log t + O(h/t)$ (see \cite{Iv1}), it follows that
\begin{equation}
  \label{mkxhIvic2}
   m_{k,\ell}(x,h)=  \int_{0}^{x} \sum_{q=1}^{\infty} \frac{c_q(h)}{q^2} P_k(t,q)P_{\ell}(t,q)    dt +O(ht^{\varepsilon}).
\end{equation}  
We now simplify the integrands in 
\eqref{mkxhIvic} and \eqref{mkxhIvic2}.  We denote them  as
\begin{equation}
  \label{qk}
  q_{k,\ell}(t,h) :=  \sum_{q=1}^{\infty} \frac{c_q(h)}{q^2} P_k(t,q) P_{\ell}(t+h,q)
\end{equation}
and
\begin{equation}
  \label{rk}
  r_{k,\ell}(t,h) := \sum_{q=1}^{\infty} \frac{c_q(h)}{q^2} P_k(t,q)P_{\ell}(t,q). 
\end{equation}
Observe that \eqref{mkxhIvic} and \eqref{mkxhIvic2} imply 
\begin{equation}
    \label{integrandids}
    \int_{0}^{x} q_{k,\ell}(t,h) dt = \int_{0}^{x} r_{k,\ell}(t,h) +O(ht^{\varepsilon}).
\end{equation}
We first calculate $q_{k,\ell}(t,h)$. Applying \eqref{Pk} twice,  it follows that 
\begin{equation}
  \label{qkthintegral}
  q_{k,\ell}(t,h) =
  \frac{1}{(2 \pi i)^2} \int_{C_2} \int_{C_1}
  \zeta^k(s_1+1) \zeta^{\ell}(s_2+1)  \mathcal{D}_{k,\ell}(s_1,s_2)
 t^{s_1} (t+h)^{s_2}  ds_1 ds_2,  
\end{equation}
$C_1 = \{ s_1 \in \mathbb{C} \ | \ |s_1|=r_1\}$, $C_2 = \{ s_2 \in \mathbb{C} \ | \ |s_2|=r_2\}$, $0 < r_1,r_2 < \frac{1}{10}$, 
and 
\begin{equation}
  \label{Dk}
  \mathcal{D}_{k,\ell}(s_1,s_2) =  \sum_{q=1}^{\infty} \frac{c_q(h) G_{k}(q,s_1+1) G_{\ell}(q,s_2+1)}{q^{2+s_1+s_2}}.
\end{equation}
We now apply the residue theorem to the inner integral in \eqref{qkthintegral}. 
For each $k \in \mathbb{N}$, there exist constants $\alpha_{j,k}$ with $j \ge 0$ such that 
\begin{equation}
  \label{zetaklaurent}
 \zeta^k(s_1+1)  = s_1^{-k} (\alpha_{0,k}+ \alpha_{1,k} s_1 + \alpha_{2,k} s_1^2 + \cdots ),  \text{ where } \alpha_{0,k} =1.
\end{equation}
Furthermore, since 
\begin{align}
  \mathcal{D}_{k,\ell}(s_1,s_2) & = \mathcal{D}_{k,\ell}^{(0,0)}(0,s_2) + \mathcal{D}_{k,\ell}^{(1,0)}(0,s_2) s_1 + \frac{1}{2} \mathcal{D}_{k,\ell}^{(2,0)}(0,s_2) s_1^2 + \cdots, \text{ and } \\
  t^{s_1} & = 1 + (\log t) s_1 + \frac{1}{2} (\log t)^2 s_1^2 + \cdots
\end{align}
it follows that
\[
   \frac{1}{2 \pi i} \int_{C_1} \zeta^k(s_1+1) \mathcal{D}_{k,\ell}(s_1,s_2) t^{s_1} ds_1
   = \sum_{\substack{i_1+i_2 +i_3 = k-1 \\ i_1,i_2,i_3 \ge 0}}  \frac{\alpha_{i_1,k} \mathcal{D}_{k,\ell}^{(i_2,0)}(0,s_2) (\log t)^{i_3}}{i_{2}! i_{3}!}. 
\]
Thus 
\[
  q_{k,\ell}(t,h) = \sum_{ \substack{i_1+i_2 +i_3 = k-1 \\ i_1,i_2,i_3 \ge 0}}  \frac{\alpha_{i_1,k} (\log t)^{i_3}}{i_{2}! i_{3}!}
   \frac{1}{2 \pi i} \int_{C_2} 
    \zeta^{\ell}(s_2+1) \mathcal{D}_{k,\ell}^{(i_2,0)}(0,s_2) (t+h)^{s_2}
      ds_2.
\]
For each value of $i_2$, a similar calculation establishes 
\[
   \frac{1}{2 \pi i} \int_{C_2} \zeta^{\ell}(s_2+1) \mathcal{D}_{k,\ell}^{(i_2,0)}(0,s_2) (t+h)^{s_2} ds_2
   = \sum_{\substack{j_1+j_2 +j_3 = \ell-1 \\ j_1,j_2,j_3 \ge 0}}  \frac{\alpha_{j_1,\ell} \mathcal{D}_{k,\ell}^{(i_2,j_2)}(0,0) (\log (t+h))^{j_3}}{j_{2}! j_{3}!}
\]
and hence
\begin{equation}
  \label{qkth}
 q_{k,\ell}(t,h) = 
  \sum_{\substack{i_1,i_2,i_3 \ge 0 \\ i_1+i_2+i_3 = k-1}}   \frac{\alpha_{i_1,k}(\log t)^{i_3}}{i_2! i_3!}
  \sum_{\substack{j_1,j_2,j_3 \ge 0 \\j_1+j_2 +j_3 = \ell-1}} \frac{\alpha_{j_1,\ell}\mathcal{D}_{k,\ell}^{(i_2,j_2)}(0,0) (\log (t+h))^{j_3}}{j_{2}! j_3!}.
\end{equation}
An analogous computation establishes
\begin{equation}
  \label{rkth}
  r_{k,\ell}(t,h) = 
  \sum_{\substack{i_1,i_2,i_3 \ge 0 \\ i_1+i_2+i_3 = k-1}}   \frac{\alpha_{i_1,k}(\log t)^{i_3}}{i_2! i_3!}
  \sum_{\substack{j_1,j_2,j_3 \ge 0 \\j_1+j_2 +j_3 = \ell-1}} \frac{\alpha_{j_1,\ell}\mathcal{D}_{k,\ell}^{(i_2,j_2)}(0,0) (\log t)^{j_3}}{j_{2}! j_3!}.
\end{equation}
Formally,  \eqref{rkth} is obtained from \eqref{qkth} by replacing each $\log(t+h)$ by $\log(t)$. 
Observe that \eqref{rkth} can be further simplified. 
Let $i=i_3+j_3$ and note that $0 \le i \le k+\ell-2$ so that
\begin{equation}
   \label{rkth2} 
   r_{k,\ell}(t,h) =  \sum_{i=0}^{k+\ell-2} \alpha_i(h)  (\log x)^{k+\ell-2-i}
\end{equation}
where $\alpha_0(h) =  \frac{\mathcal{D}_{k,\ell}(0,0)}{(k-1)!(\ell-1)!}$, the term arising from 
$(i_1,i_2,i_3) =(0,0,k-1)$ and $(j_1,j_2,j_3)= (0,0,\ell-1)$.
We now show that $
\mathcal{D}_{k,\ell}(0,0)=c_{k,\ell}(h)=C_{k,\ell} f_{k,\ell}(h)$. 
This will be deduced from the following lemma.  This lemma will also be used in our proof of Theorem \ref{mainthm}. 
\begin{lemma} \label{f1f2prod}
Let $f_1, f_2$ be nonzero multiplicative functions,  $\tau_1, \tau_2$ be real numbers, and $h$ a
natural number such that 
\begin{equation}
   \label{Stau1tau2}
   \mathcal{S}(\tau_1,\tau_2;h) := \sum_{g \mid h} \frac{1}{g^{\tau_1}} \sum_{d=1}^{\infty} \frac{\mu(d) f_1(gd)f_2(gd)}{d^{\tau_2}}
\end{equation}
is absolutely convergent.  \\
(i) We have 
\begin{equation}
  \label{Sproduct0}
   \mathcal{S}(\tau_1,\tau_2;h) := \sum_{g \mid h} \frac{1}{g^{\tau_1}}
    \prod_{(p,g)=1} \Big ( 1 - \frac{f_1(p)
         f_2(p)}{p^{\tau_2}} \Big ) 
          \prod_{p^{\alpha} \mid \mid g} 
\Big( 
f_1(p^{\alpha})f_2(p^{\alpha}) - \frac{f_1(p^{\alpha+1})f_2(p^{\alpha+1})  }{ p^{\tau_2}}
\Big). 
\end{equation}
(ii)  If for every prime $p$,  $f_1(p)f_2(p) \ne p^{\tau_2}$,  then
\begin{equation}
   \label{Sproduct}
   \mathcal{S}(\tau_1,\tau_2;h) = 
   \prod_{p} \Big ( 1 - \frac{f_1(p)
         f_2(p)}{p^{\tau_2}} \Big )    
\prod_{p^{\alpha} \mid \mid h} 
\sum_{j=0}^{\alpha}  \Big( 
 \frac{f_1(p^{j}) f_2(p^{j})}{p^{\tau_1 j}} - \frac{f_1(p^{j+1})
         f_2(p^{j+1})}{ p^{\tau_1 j + \tau_2}}
\Big)  \Big ( 1 - \frac{f_1(p)
         f_2(p)}{p^{\tau_2}} \Big )^{-1}. 
\end{equation}
\end{lemma}
The proof of this lemma is deferred to the end of the section. 
Using this lemma, we shall demonstrate 
\begin{equation}
  \label{Dk00identity}
\mathcal{D}_{k,\ell}(0,0)=c_{k,\ell}(h)=C_{k,\ell} f_{k,\ell}(h).
\end{equation} 
Inserting the identity  $c_q(h)=\sum_{d \mid h, d \mid q} d \mu(q/d)$ in \eqref{Dk}, 
exchanging summation, and making the variable change $q \to qd$ leads to 
\[
   \mathcal{D} _{k,\ell}(s_1,s_2) = \sum_{d \mid h} \frac{1}{d^{1+s_1+s_2}} 
    \sum_{q=1}^{\infty} \frac{\mu(q) G_{k}(qd,s_1+1) G_{\ell}(qd,s_2+1)}{q^{2+s_1+s_2}}.
\]
This is now in the form of the previous lemma.   We set $s_1=s_2=0$,  $f_1(n)= G_{k}(n,1)$, $f_2(n) = G_{\ell}(n,1)$, $\tau_1= 1$,
and $\tau_2=2$, to obtain
\begin{equation}
 \label{Dkprod}
  \mathcal{D}_{k,\ell}(0,0) = 
  \prod_{p} 
     \mathcal{E}(p,0)
      \prod_{p^\alpha \mid \mid h}  \frac{\sum_{j=0}^{\alpha}  \mathcal{E}(p,j)}{     \mathcal{E}(p,0)},     
\end{equation}
where 
\begin{equation}
    \label{Ekid}
     \mathcal{E}(p,j) :=    \frac{G_{k}(p^j,1)G_{\ell}(p^j,1)  }
            {p^{j}} - \frac{G_{k}(p^{j+1},1)G_{\ell}(p^{j+1},1)}{p^{j+2}}.
\end{equation}
We now show that $G_{k}(p^j,1) = \sigma_{k-1}(p^j,1)$.
Observe that by \cite[p. 592]{CG}
\begin{equation}
   \label{Gksids}
   G_{k}(p^{j},1) =  \Big(1-\frac{1}{p} \Big)^{-1}( \sigma_{k}(p^{j},1) - \sigma_{k}(p^{j-1},1)). 
\end{equation}
Thus by \eqref{sigmapjid} and \eqref{recurrence},   
\begin{align*}
  G_{k}(p^{j},1) &  
   =  \Big( 1-\frac{1}{p} \Big)^{-1}  \Big( 1- \frac{1}{p} \Big)^{k} \sum_{i=0}^{\infty} \frac{ \tau_k(p^{j+i}) - \tau_k(p^{j-1+i}) }{p^{is}} 
   = \Big( 1- \frac{1}{p} \Big)^{k-1} \sum_{i=0}^{\infty} \frac{ \tau_{k-1}(p^{j+i})  }{p^{is}}   
  = \sigma_{k-1}(p^j,1), 
\end{align*}
by definition.  Hence, 
\begin{equation}
  \label{Ekpj}
   \mathcal{E}(p,j) :=    \frac{\sigma_{k-1}(p^j,1)\sigma_{\ell-1}(p^j,1) }
            {p^{j}} - \frac{\sigma_{k-1}(p^{j+1},1)\sigma_{\ell-1}(p^{j+1},1) }{p^{j+2}}.
\end{equation}
Observe that 
\begin{equation}
\begin{split}
  \label{Ekp0}
   \mathcal{E}(p,0) & =1-\frac{ \sigma_{k-1}(p,1)  \sigma_{\ell-1}(p,1)  }{p^2}
 = 1 - (1- (1-p^{-1})^{k-1}))(1- (1-p^{-1})^{\ell-1})) \\
 &  = (1-p^{-1})^{k-1}+(1-p^{-1})^{\ell-1} - (1-p^{-1})^{k+\ell-2} 
\end{split}
\end{equation}
Thus, by \eqref{Dkprod}, \eqref{Ekpj}, and 
\eqref{Ekp0}, 
\[
   \mathcal{D}_{k,\ell}(0,0) = 
  C_{k,\ell}
      \prod_{p^\alpha \mid \mid h}  
      \frac{  \sum_{j=0}^{\alpha}  \frac{\sigma_{k-1}(p^j,1)\sigma_{\ell-1}(p^j,1) }
            {p^{j}} - \frac{\sigma_{k-1}(p^{j+1},1)\sigma_{\ell-1}(p^{j+1},1) }{p^{j+2}}}{
             (1-p^{-1})^{k-1}+(1-p^{-1})^{\ell-1} - (1-p^{-1})^{k+\ell-2}}.
\]
In order to show $\mathcal{D}_{k,\ell}(0,0) =c_{k,\ell}(h)$ it suffices to show the last product 
equals $f_{k,\ell}(h)$.   
By \eqref{fklpa} this is equivalent to showing 
\begin{equation} 
\begin{split}
 \label{keyidentity}
     & \Big(1-\frac{1}{p} \Big)^{-k-\ell+2} \sum_{j=0}^{\alpha} 
    \Big( \frac{ \sigma_{k-1}(p^j,1)\sigma_{\ell-1}(p^j,1)  }{p^{j}} 
    - \frac{\sigma_{k-1}(p^{j+1},1)\sigma_{\ell-1}(p^{j+1},1)}{p^{j+2}} \Big) \\
         & =
     1 + \sum_{i=1}^{\alpha} (\tau_k(p^i)\tau_{\ell}(p^i)-\tau_k(p^{i-1})\tau_{\ell}(p^{i-1}))X^i 
    +    \sum_{i=\alpha+1}^{\infty}(\tau_k(p^{\alpha}) \tau_{\ell-1}(p^i)X^i  
   + \tau_{\ell}(p^{\alpha})  \tau_{k-1}(p^i))X^i.
\end{split}
\end{equation} 
We denote this identity as $\mathcal{L}_{k,\ell}(\al)= \mathcal{R}_{k,\ell}(\al)$.  We prove this by 
induction on $\alpha$.   First, a calculation shows that 
\begin{equation}
 \label{L1R1}
  \mathcal{L}_{k,\ell}(1)= \mathcal{R}_{k,\ell}(1) = 
  k (1-X)^{-(\ell-1)}+\ell (1-X)^{-(k-1)} - (k-1)(\ell-1)X - k - \ell +1. 
\end{equation}
Assume that $\mathcal{L}_{k,\ell}(\al)= \mathcal{R}_{k,\ell}(\al)$ for $\al \in \mathbb{N}$.  
We aim to show that $\mathcal{L}_{k,\ell}(\alpha+1) = \mathcal{R}_{k,\ell}(\alpha+1)$.
 To simplify notation we set $X=\frac{1}{p}$.  
 Observe that 
\begin{align*}
\mathcal{L}_k(\alpha+1) - \mathcal{L}_k(\alpha) 
    & =  (\tau_k(p^{\alpha+1}) \tau_{\ell}(p^{\alpha+1})- \tau_k(p^{\alpha}) \tau_{\ell}(p^{\alpha})
    -\tau_k(p^{\alpha}) \tau_{\ell-1}(p^{\alpha+1})-\tau_{\ell}(p^{\alpha}) \tau_{k-1}(p^{\alpha+1})
    )X^{\alpha+1} \\
    & +  (\tau_k(p^{\alpha+1})- \tau_k(p^{\alpha})) \sum_{i=\alpha+2}^{\infty} \tau_{\ell-1}(p^i) X^i 
     +  (\tau_{\ell}(p^{\alpha+1})- \tau_{\ell}(p^{\alpha})) \sum_{i=\alpha+2}^{\infty} \tau_{k-1}(p^i) X^i\\
      & =   (\tau_k(p^{\alpha+1}) \tau_{\ell}(p^{\alpha+1})- \tau_k(p^{\alpha}) \tau_{\ell}(p^{\alpha})
    -\tau_k(p^{\alpha}) \tau_{\ell-1}(p^{\alpha+1})-\tau_{\ell}(p^{\alpha}) \tau_{k-1}(p^{\alpha+1})
    )X^{\alpha+1} \\
    & +  \tau_{k-1}(p^{\alpha+1}) \sum_{i=\alpha+2}^{\infty} \tau_{\ell-1}(p^i) X^i 
     +   \tau_{\ell-1}(p^{\alpha+1}) \sum_{i=\alpha+2}^{\infty} \tau_{k-1}(p^i) X^i.
\end{align*}
Next notice that we can simplify the coefficient of $X^{\alpha+1}$.  Observe that 
\[
   \tau_{k-1}(p^{\alpha+1}) \tau_{\ell-1}(p^{\alpha+1})= 
\tau_k(p^{\alpha+1}) \tau_{\ell}(p^{\alpha+1})- \tau_k(p^{\alpha}) \tau_{\ell}(p^{\alpha})
    -\tau_k(p^{\alpha}) \tau_{\ell-1}(p^{\alpha+1})-\tau_{\ell}(p^{\alpha}) \tau_{k-1}(p^{\alpha+1}).
\]
Rearranging, this is if and only if 
\[
   \tau_k(p^{\alpha+1}) \tau_{\ell}(p^{\alpha+1})= 
    \tau_{k-1}(p^{\alpha+1}) \tau_{\ell-1}(p^{\alpha+1})+
    \tau_k(p^{\alpha}) \tau_{\ell}(p^{\alpha})
    +\tau_k(p^{\alpha}) \tau_{\ell-1}(p^{\alpha+1})+\tau_{\ell}(p^{\alpha}) \tau_{k-1}(p^{\alpha+1}).
\]
Using \eqref{taukpj}  this is
\begin{align*}
   & \binom{k+\al}{\al+1} \binom{\ell+\al}{\al+1} = \\
  & \binom{k-1+\al}{\al+1}\binom{\ell-1+\al}{\al+1}+
   \binom{k+\al-1}{\al} \binom{\ell+\al-1}{\al} + \binom{k+\al-1}{\al} \binom{\ell +\al-1}{\al+1} + \binom{k+\al-1}{\al+1}
   \binom{\ell+\al-1}{\al}. 
\end{align*}
However, this last identity follows from two applications of Pascal's identity. Thus 
\begin{equation}
  \label{Lkdifference}
 \mathcal{L}_k(\alpha+1) - \mathcal{L}_k(\alpha)   =    \tau_k(p^{\alpha+1}) \tau_{\ell}(p^{\alpha+1})X^{\alpha+1}  +   \sum_{i=\alpha+2}^{\infty} (\tau_{k-1}(p^{\alpha+1}) \tau_{\ell-1}(p^i) 
     +   \tau_{\ell-1}(p^{\alpha+1})  \tau_{k-1}(p^i)) X^i.
\end{equation}
We now calculate $ \mathcal{R}_{k,\ell}(\al+1)- \mathcal{R}_{k,\ell}(\al)$. Observe that 
\begin{align*}
  & \mathcal{R}_{k,\ell}(\al+1)- \mathcal{R}_{k,\ell}(\al) =
  \Big(1- \frac{1}{p} \Big)^{-k-\ell+2}  X^{\al+1}
    \Big(  \sigma_{k-1}(p^{\alpha+1},1) \sigma_{\ell-1}(p^{\alpha+1},1)  - 
    \frac{\sigma_{k-1}(p^{\al+2},1)\sigma_{\ell-1}(p^{\al+2},1)}{p^2}  \Big) \\
  & =  \Big(1- \frac{1}{p} \Big)^{-k-\ell+2}  X^{\al+1} \cdot \\
  & \Big(
    \Big( \sigma_{k-1}(p^{\alpha+1},1)- \frac{ \sigma_{k-1}(p^{\alpha+2},1)}{p} \Big)  \sigma_{\ell-1}(p^{\alpha+1},1)  
    +   \frac{ \sigma_{k-1}(p^{\alpha+2},1)}{p}  \Big( \sigma_{\ell-1}(p^{\al+1},1)
    \frac{\sigma_{\ell-1}(p^{\al+2},1)}{p}  \Big)
  \Big).
\end{align*}
However, 
\begin{align*}
 \sigma_{k-1}(p^{\al+1},1)  - \frac{\sigma_{k-1}(p^{\al+2},1)}{p}  &  = \Big( 1-\frac{1}{p} \Big)^{k-1}
    \Big( \sum_{i=0}^{\infty} \frac{\tau_{k-1}(p^{\al+1+i})}{p^i}
  - \sum_{i=0}^{\infty} \frac{\tau_{k-1}(p^{\al+2+i})}{p^{i+1}} \Big) \\
  & = \Big( 1-\frac{1}{p} \Big)^{k-1} \tau_{k-1}(p^{\al+1})
\end{align*}
and thus
\begin{align*}
   \mathcal{R}_{k,\ell}(\al+1)- \mathcal{R}_{k,\ell}(\al) & =
      X^{\al+1} 
   \Big(
    \tau_{k-1}(p^{\al+1})
     \sigma_{\ell-1}(p^{\alpha+1},1)  
    +  \tau_{\ell-1}(p^{\al+1}) \frac{ \sigma_{k-1}(p^{\alpha+2},1)}{p}  
      \Big) \\
      & =   X^{\al+1} 
   \Big(
    \tau_{k-1}(p^{\al+1})
     \sum_{i=0}^{\infty} \frac{\tau_{\ell-1}(p^{\al+i+1})}{p^i}
    +  \tau_{\ell-1}(p^{\al+1})
    \frac{1}{p}  \sum_{i=0}^{\infty} \frac{\tau_{\ell-1}(p^{\al+2+1})}{p^i}
      \Big) \\
      & =   \tau_{k-1}(p^{\al+1})  \tau_{\ell-1}(p^{\al+1}) X^{\al+1} 
      + \sum_{i=\alpha+2}^{\infty}(\tau_{k-1}(p^{\alpha+1}) \tau_{\ell-1}(p^i) 
     +   \tau_{\ell-1}(p^{\alpha+1})  \tau_{k-1}(p^i))   X^i \\
     & =  \mathcal{L}_{k,\ell}(\al+1)- \mathcal{L}_{k,\ell}(\al), 
\end{align*}
by \eqref{Lkdifference}.  Hence, by the induction hypothesis $  \mathcal{L}_{k,\ell}(\al+1)= \mathcal{R}_{k,\ell}(\al+1)$ as desired.   Thus we have $\mathcal{L}_{k,\ell}(\al)= \mathcal{R}_{k,\ell}(\al)$ for all $\al \in \mathbb{N}$.
Consequently, we have proven \eqref{Dk00identity}.

In summary, we arrive at the following conjecture. 
\begin{conj} (Additive divisor conjecture) \label{adddivconj}
 Let  $k,\ell>1$ be a natural number and $x>0$ is large.  Then 
 there exists  a positive constant $\theta_{k,\ell} \in [\tfrac{1}{2},1)$ such that 
\begin{equation}
  \label{Dkxhconj}
  D_{k,\ell}(x,h) =  
  \int_{0}^{x} q_{k,\ell}(t,h) dt +  E_{k,\ell}(x,h),
\end{equation}
where $q_{k,\ell}(t,h)$ is given by \eqref{qkth} and for every $\varepsilon >0$
\begin{equation}
  \label{Ekxh}
  E_{k,\ell}(x,h) \ll x^{\vartheta_{k,\ell}+\varepsilon} \text{ uniformly for } 1 \le h \le x^{1-\varepsilon}.
\end{equation}
Moreover,  in \eqref{qkth}, the coefficient of $\log(t)\log(t+h)$ is  $\mathcal{D}_{k,\ell}(0,0) = C_{k,\ell} f_{k,\ell}(h)$
where $C_{k,\ell}$ is given by \eqref{Ckl} and $f_{k,\ell}(h)$ is the multiplicative function defined by \eqref{fklpa}.  
\end{conj}
To abbreviate notation we set $E_k(x,h)=E_{k,k}(x,h)$, $q_k(x,h) = q_{k,k}(x,h)$, and $\vartheta_{k}
=\vartheta_{k,k}$.  \\ 
\noindent {\bf Remarks}. 
\begin{enumerate}
\item It appears that Titchmarsh \cite{Tit2} was the first to conjecture the leading term in the asymptotic formula for  a weighted version of
$D_{3}(x,1)$, based on the circle method.   Vinogradov \cite{Vin} proposed the general form of a conjectural formula for $D_{k}(x,h)$ (see equation (2) of \cite{Vin}).  However, few
details were given and he did not provide any formulae for the coefficients of $q_{k}(t,h)$.  
Then in the nineties Ivi\'{c} \cite{Iv1}, \cite{Iv2} and Conrey-Gonek \cite{CG}
provided more precise formulae following Duke, Friedlander, and Iwaniec's $\delta$-method. 
\item It is not clear what is the true size of the error term $E_{k}(x,h)$ and various opinions have been expressed.
 Conrey and Gonek \cite{CG} conjectured that 
 $\vartheta_k= \frac{1}{2}$ in the case that $h \le \sqrt{x}$.  However, 
 Conrey and Keating \cite{CK} revised this  to  $\vartheta_k= \frac{1}{2}$ is 
 valid for all $h \le x^{1-\varepsilon}$. 
 Recent work of Farzad Aryan suggests that $\vartheta_2 =\frac{1}{2}$ 
 is the correct value.  In fact, Aryan \cite{Ar} shows that a smoothed variant of $D_2(x,h)$ has error term $O(x^{\frac{1}{2}+\varepsilon} h^{\alpha})$
where $\alpha$ is given by \eqref{Ramanujan} and conjecturally $\alpha=0$. 
On the other hand, Vinogradov \cite{Vin} conjectured that $E_{k}(x,h) \ll x^{1-\frac{1}{k}}$ in the case of $h$ {\it fixed}. 
 Ivi\'{c} \cite{Iv1} suggested that Vinogradov's bound was slightly too strong and that perhaps  
 $E_{k}(x,h) \ll x^{1-\frac{1}{k}} (\log x)^{C_k'}$ for a positive constant $C_k'$.  In light of these diverging opinions, 
 it would be beneficial to have numerical data checking this conjecture.  
\item   Note that the conjecture is sometimes written as 
\begin{equation}
   \label{Dkxhconj2}
  D_{k,\ell}(x,h) =  
  \int_{0}^{x} r_{k,\ell}(t,h) dt +  \tilde{E}_{k,\ell}(x,h),
\end{equation}
where $ \tilde{E}_{k,\ell}(x,h) \ll x^{\vartheta_{k,\ell}+\varepsilon} \text{ uniformly for } 1 \le h \le x^{1-\varepsilon}$.
By \eqref{integrandids} we may replace $q_{k,\ell}(t,h)$ by $r_{k,\ell}(t,h)$ with an error 
$O(x^{\frac{1}{2}+\varepsilon})$ for $h \le \sqrt{x}$.  Since we expect that $\theta_{k,\ell} \ge \frac{1}{2}$, 
it should not matter whether the main term in \eqref{Dkxhconj} or \eqref{Dkxhconj2} is used for $h \le \sqrt{x}$.
However, as we expect to have an asymptotic formula for $h \le x^{1-\varepsilon}$, 
it is preferable to use the form \eqref{Dkxhconj}.  
\item   In the case $k=\ell=2$, this conjecture agrees with Ingham's result \eqref{Ingham}. 
Note that $2 (1-\frac{1}{p})- (1-\frac{1}{p})^2  = 1-p^{-2}$
and thus 
$
  C_2  
  =  \prod_{p} (1-p^{-2}) =
  \frac{6}{\pi^2}$. 
Also, $\tau_{1}(p^j)=\tau_{1}(p^{j+1})=1$, $H_{1,j}(u)=H_{1,j+1}(u)=1$, and by \eqref{fklpa}
$f_2(p^{\alpha}) = \frac{\sum_{j=0}^{\alpha} ( \frac{1}{p^j} - \frac{1}{p^{j+2}} )}{1-p^{-2}}
   =\sum_{j=0}^{\alpha} \frac{1}{p^j} = \sigma_{-1}(p^{\alpha})$.
 \item Recently, Andrade, Bary-Soker, and Rudnick \cite{An} proved a function field version of the above conjecture. 
\item  Although  Conjecture \ref{adddivconj} remains open for $k \ge 3$, averaged versions 
have been established.  For instance, see \cite{BBMZ} and \cite{IW}. 
Recently, Matom\"aki, M.  Radziwi\l\l, and Tao \cite{Mat} have established an almost all result.  They have shown 
that  there exists $C_k >0$  such that if $x \ge H \ge (\log x)^{C_k} \ge 2$, then 
\[
  D_{k}(2x,h)-D_{k}(x,h) = \Big( \int_{x}^{2x} q_{k}(t,h) \, dt  \Big) (1+o(1))
\]
for all but $o(H)$ values of $|h| \le H$.  
\end{enumerate}

To complete this section we provide the proof of Lemma \ref{f1f2prod}.
\begin{proof}[Proof of Lemma \ref{f1f2prod}]
For each $g \mid h$, write $g = \prod_{p^{\alpha \mid \mid g}} p^{\alpha}$.
By multiplicativity of the inner summand it follows that 
\[
     \mathcal{S}(\tau_1,\tau_2;h) := \sum_{g \mid h} \frac{1}{g^{\tau_1}}  \Big(
     \prod_{(p,g)=1} \sum_{m=0}^{\infty} \frac{\mu(p^m) f_1(p^m) f_2(p^m)}{(p^m)^{\tau_2}}
     \Big) \cdot 
     \Big(
     \prod_{p^{\alpha} \mid \mid g} 
     \sum_{m=0}^{\infty} \frac{\mu(p^m) f_1(p^{m+\alpha}) f_2(p^{m+\alpha})}{(p^m)^{\tau_2}}
     \Big).
\]
Simplifying this expression, using that $\mu(1)=1, \mu(p)=-1$, and $\mu(p^m)=0$ for $m \ge 2$, 
\[
   \mathcal{S}(\tau_1,\tau_2;h) := \sum_{g \mid h} \frac{1}{g^{\tau_1}}
    \prod_{(p,g)=1} \Big ( 1 - \frac{f_1(p)
         f_2(p)}{p^{\tau_2}} \Big ) 
          \prod_{p^{\alpha} \mid \mid g} 
\Big( 
f_1(p^{\alpha})f_2(p^{\alpha}) - \frac{f_1(p^{\alpha+1})f_2(p^{\alpha+1})  }{ p^{\tau_2}}
\Big) 
\]
Since $f_1(p)f_2(p) \ne p^{\tau_2}$, we multiply and divide each summand by $\prod_{p \mid g} \Big ( 1 - \frac{f_1(p)
f_2(p)}{p^{\tau_2}} \Big )$ to obtain 
\begin{align}
 \mathcal{S}(\tau_1,\tau_2;h) & := \sum_{g \mid h} \frac{1}{g^{\tau_1}}
    \prod_{p} \Big ( 1 - \frac{f_1(p)
         f_2(p)}{p^{\tau_2}} \Big ) 
          \prod_{p^{\alpha} \mid \mid g} 
\Big( 
f_1(p^{\alpha})f_2(p^{\alpha}) - \frac{f_1(p^{\alpha+1})f_2(p^{\alpha+1})  }{ p^{\tau_2}}
\Big)  \Big ( 1 - \frac{f_1(p)
         f_2(p)}{p^{\tau_2}} \Big )^{-1}  \nonumber \\
      &  =  \prod_{p} \Big ( 1 - \frac{f_1(p)
         f_2(p)}{p^{\tau_2}} \Big ) 
    \sum_{g \mid h} \frac{1}{g^{\tau_1}}
    \prod_{p^{\alpha} \mid \mid g} 
\Big( 
f_1(p^{\alpha})f_2(p^{\alpha}) - \frac{f_1(p^{\alpha+1})f_2(p^{\alpha+1})  }{ p^{\tau_2}}
\Big)  \Big ( 1 - \frac{f_1(p)
         f_2(p)}{p^{\tau_2}} \Big )^{-1} \label{rpalpha}.
\end{align}
Let $r$ be a multiplicative function defined on prime powers by 
\[
  r(p^{\alpha}) = \Big( 
f_1(p^{\alpha})f_2(p^{\alpha}) - \frac{f_1(p^{\alpha+1})f_2(p^{\alpha+1})  }{ p^{\tau_2}}
\Big)  \Big ( 1 - \frac{f_1(p)
         f_2(p)}{p^{\tau_2}} \Big )^{-1}.
\]
The sum in \eqref{rpalpha} equals $\sum_{g \mid h} r(g)g^{-\tau_1}$.
By multiplicativity, 
$$\sum_{g \mid h} r(g) g^{-\tau_1} 
   = \prod_{p^{\alpha} \mid \mid h}  \sum_{j=0}^{\alpha} r(p^j)p^{-j \tau_1}
   = \prod_{p^{\alpha} \mid \mid h}  \sum_{j=0}^{\alpha}  \Big( 
 \frac{f_1(p^{j}) f_2(p^{j})}{p^{\tau_1 j}} - \frac{f_1(p^{j+1})
         f_2(p^{j+1})}{ p^{\tau_1 j + \tau_2}}
\Big)  \Big ( 1 - \frac{f_1(p)
         f_2(p)}{p^{\tau_2}} \Big )^{-1}.
$$
Inserting this expression in \eqref{rpalpha} we derive \eqref{Sproduct}.
\end{proof}

\section{A lower bound for $D_{k,\ell}(x,h)$} \label{Lowerbound}
In this section, we establish Theorem \ref{mainthm}, which provides a lower bound for $D_{k,\ell}(x,h)$. 
Before proving this result, we require a proposition which
gives an asymptotic estimate for a certain divisor sum. 
\begin{proposition} \label{selberg}
Let $k,\ell \in \mathbb{N}$ and $k,\ell \ge 2$. \\
(i) Then there exists $h_0 = h_0(k,\ell)>0$ such that for $h \ge h_0$, 
\begin{equation}
\begin{split}
   \label{case1}
   &  \sum_{\substack{a,b \le X \\ (a,b) \mid h}} \frac{ \tau_{k}(a)\tau_{\ell}(b)}{[a,b]} 
     = \frac{\tilde{C}_{k,\ell} g_{k,\ell}(h)}{ k!  \ell ! } (\log X)^{k+ \ell} +
    O_{k,\ell} \Big( \prod_{p \mid h}(1+p^{-1})^{k \ell }  (\log X)^{k+ \ell -1} \log \log h) \Big)  \\
    & + O_{k,\ell} \Big(  
    \exp \Big(  (12.94 m^4-12.81 m^3 +4.52 m^2 )
    \frac{(\log h)^{1-0.99/m}}{ \log \log h}  \Big)
    (\log X)^{2 m} X^{-0.99/M}  \Big)
\end{split}
\end{equation}
where $m=\min(k,\ell)$, $M=\max(k,\ell)$,
\begin{equation}
  \label{tilCell}
  \tilde{C}_{k,\ell}  = \prod_p \Big (  \Big
       (1-\frac{1}{p} \Big )^{k}
       +(1-\frac{1}{p} \Big )^{\ell}
        -
       \Big ( 1-\frac{1}{p} \Big )^{k+\ell} \Big ),
\end{equation}
and $g_{k,\ell}$ is the multiplicative function defined on prime powers by 
\begin{equation}
 \label{gldef}
  g_{k,\ell}(p^{\alpha}) :=   
    \sum_{j=0}^{\alpha} 
    \Big( \frac{\sigma_{k}(p^j,1)\sigma_{\ell}(p^j,1)  }
            {p^{j}} - \frac{\sigma_{k}(p^{j+1},1)\sigma_{\ell}(p^{j+1},1)}{p^{j+2}} \Big)
  \Big(   \Big(1-\frac{1}{p} \Big)^{k} +
       \Big(1-\frac{1}{p} \Big)^{\ell} - 
       \Big( 1-\frac{1}{p} \Big )^{k+\ell} 
       \Big)^{-1}.
\end{equation}
(ii) If $1 \le h < h_0$, then the same result holds as in equation \eqref{case1} except the second $O_{k,\ell}$ term
in this equation
is replaced by $O_{k,\ell}( (\log X)^{2 k} X^{-0.99/M})$ where $O_{k,\ell}$ constant is polynomial in 
$k$ and $\ell$.   
%
\end{proposition}
We have not tried to obtain the best possible error term here. 
Note that the sum in this proposition bears some resemblance to the quadratic forms that occur in the standard 
Selberg sieve \cite{Se}.   A similar sum is studied in \cite{DIT}.

With these two results in hand, we prove our main result.
\begin{proof}[Proof of Theorem \ref{mainthm}]
Let $x > 1$.
For the lower bound, we make use of the identity
\begin{equation}
   \label{lbid}
   \tau_k(m) \ge \sum_{\substack{d \mid m \\ d \le \sqrt{x}}} \tau_{k-1}(d) \text{ for } m \ge x.
\end{equation}
It follows that  
\begin{align*}
 \sum_{x \le n \le 2x} \tau_k(n) \tau_{\ell}(n+h) & \ge 
\sum_{x \le n \le 2x} \sum_{\substack{a \mid n \\ a \le \sqrt{x}}} \tau_{k-1}(a) \sum_{\substack{b \mid n+h \\ b \le \sqrt{x+h}}} \tau_{\ell-1}(b) \\
& = \sum_{a \le \sqrt{x}} \sum_{b \le \sqrt{x+h}} \tau_{k-1}(a) \tau_{\ell-1}(b) \sum_{\substack{x \le n \le 2x \\ a \mid n \\ b \mid n+h}} 1.
\end{align*} If $(a,b) \mid h$, the inner sum is $\frac{x}{[a,b]} + O(1)$ and 
otherwise it is 0. 
Thus 
\begin{align*}
  \sum_{x \le n \le 2x} \tau_k(n) \tau_{\ell}(n+h) & \ge \sum_{a \le \sqrt{x}} \sum_{\substack{b \le \sqrt{x} \\ (a,b) \mid h}} \tau_{k-1}(a)\tau_{\ell-1}(b) \Big( \frac{x}{[a,b]} + O(1) \Big). 
\end{align*}
The $O(1)$ term contributes
$$ \Big(\sum_{a \le \sqrt{x}} \tau_{k-1}(a) \Big)
\Big(\sum_{a \le \sqrt{x}} \tau_{\ell-1}(a) \Big) \ll  ( \sqrt{x} (\log x)^{k-2}) ( \sqrt{x} (\log x)^{\ell-2})
 \ll x (\log x)^{k+\ell-4}$$
and by 
Proposition \ref{selberg} with $X =\sqrt{x}$
\begin{equation}
\begin{split}
   \label{specialsum}
    & \sum_{\substack{a,b \le \sqrt{x} \\ (a,b) \mid h}} \frac{ \tau_{k-1}(a)\tau_{\ell-1}(b)}{[a,b]}  =
     \frac{\tilde{C}_{k-1,\ell-1} g_{k-1,\ell-1}(h)}{(k-1)! (\ell-1)! 2^{k+\ell-2}}   (\log x)^{k+\ell-2} \\
     & +O_{k} \Big(g_{k-1,\ell-1}(h) (\log x)^{k+\ell-3} \log \log h
     + \exp\Big( \frac{\mathcal{C}_{k,\ell} (\log h)^{\vartheta}}{\log \log h}   \Big) \frac{ (\log x)^{2m-2}}{ x^{\beta}}
     \Big)
\end{split}
\end{equation}
where $\tilde{C}_{k-1,\ell-1}$ is defined by \eqref{tilCell}, 
$\vartheta=1-\frac{0.99}{m-1}$, 
and $\beta=\frac{0.495}{m-1}$, and $\mathcal{C}_{k,\ell}$ is a positive constant depending on $k$ and $\ell$.  It may be checked that $g_{k-1,\ell-1}(h) \gg_{k,\ell} 1$ for all $h \in \mathbb{N}$. 
The second error term in \eqref{specialsum} is dominated by the first if and only if 
$x^{\beta}(\log x)^{k+\ell-2m-1}  \gg \frac{\exp(\mathcal{C}_{k,\ell} \frac{(\log h)^{\vartheta}}{\log \log h} ) }{\log \log h}$.
In other words, 
\begin{equation}
   \label{desiredinequality}
    \exp(\beta \log x +(k+\ell-2m-1)\log_2 x) \gg 
    \exp\Big(\frac{\mathcal{C}_{k,\ell}(\log h)^{\vartheta}}{\log \log h} - \log_3 h  \Big). 
\end{equation}
This inequality will hold if we impose the condition $\frac{\mathcal{C}_{k,\ell}(\log h)^{\vartheta}}{\log \log h} 
\le \frac{\beta}{2} \log x$.   This implies that $\log_2 h \ll \log_2 x$.   Therefore 
$\mathcal{C}_{k,\ell}(\log h)^{\vartheta}
\le \frac{\beta}{2} \log x \log \log x$. Solving for $h$ we find that 
$h \le \exp(\tilde{B}_{k,\ell} (\log x \log \log x)^{\frac{1}{\vartheta}})$ for some positive $\tilde{B}_{k,\ell}$. 
Combining the above,
\begin{equation}
     \label{dyadic}
      \sum_{x \le n \le 2x} \tau_k(n) \tau_{\ell}(n+h) \ge 
     \frac{  \tilde{C}_{k-1,\ell-1} g_{k-1,\ell-1}(h)}{(k-1)! (\ell-1)! 2^{k+\ell-2}}  x(\log x)^{k+\ell-2}
     + O_{k}(g_{k-1}(h)x (\log x)^{2k-3} \log \log h),
\end{equation}
as long as $h \le \exp(\tilde{B}_{k,\ell} (\log x \log \log x)^{\frac{m-1}{m-1.99}})$.
Now split the interval $[\sqrt{x},x]$ into $O(\log x)$ dyadic intervals and apply \eqref{dyadic} to obtain
\begin{equation}
  \label{lowerbound}
  \sum_{\sqrt{x} \le n \le x} \tau_k(n) \tau_{\ell}(n+h)  \ge
   \frac{ C_{k,\ell} f_{k,\ell}(h)}{(k-1)! (\ell-1)! 2^{k+\ell-2}}  x(\log x)^{k+\ell-2}
    + O_{k}(f_{k,\ell}(h)x (\log x)^{2k-3} \log \log h)
\end{equation}
valid for 
$h \le \exp(B_{k,\ell} (\log x \log \log x)^{\frac{m-1}{m-1.99}})$  for another positive constant $B_{k,\ell}$, 
where $C_{k,\ell} := \tilde{C}_{k-1,\ell-1}$ and $f_{k,\ell}(h) := g_{k-1,\ell-1}(h)$.  
Since $\tau_k$  and $\tau_{\ell}$ are positive functions, 
we establish the theorem. 
\end{proof}
\noindent {\bf Remark}.  The above argument in the case $k=2$ yields an asymptotic formula for 
$D_2(x,h)$.  This  is essentially the argument Ingham used in \cite{In} and \cite{In2} to obtain first
an upper bound and then an asymptotic for $D_2(x,h)$. 

We have reduced the proof of Theorem \ref{mainthm}  to a verification of Proposition \ref{selberg}.
Not surprisingly, we must understand the double Dirichlet series 
\[
    \mathcal{A}(s_1,s_2) = \sum_{\substack{ a,b \ge 1 \\ (a,b) \mid h}} \frac{\tau_{k}(a) \tau_{\ell}(b)}{[a,b] a^{s_1} b^{s_2}}. 
\]
We shall show that $  \mathcal{A}(s_1,s_2) =\zeta(s_1+1)^{k} \zeta(s_2+1)^{\ell}   \mathcal{B}(s_1,s_2)$ where
\begin{equation}
  \label{B}
   \mathcal{B}(s_1,s_2) =  \sum_{g \mid h} \frac{1}{g^{s_1+s_2+1}} 
     \sum_{d=1}^{\infty} \frac{\mu(d) \sigma_{k}(gd,s_1+1) \sigma_{\ell}(gd,s_2+1)}{d^{s_1+s_2+2}}
\end{equation}
and we recall that $\sigma_{k}$ is the multiplicative function defined by 
$\sigma_{k}(n,s)= \Big( \sum_{a=1}^{\infty} \frac{\tau_{k}(na)}{a^s}  \Big) \zeta(s)^{-k}$.
(Some properties of $\sigma_{k}$ are listed in subsection \ref{divisorproperties}.)
We require the following bounds on $\mathcal{B}(s_1,s_2)$. 
\begin{lemma} \label{Blemma}
For $z \in \mathbb{C}$ and $h \in \mathbb{N}$, set 
\begin{equation}
  \label{Thetadefn}
 \Theta(z,h) = \prod_{p \mid h} (1 + p^{-z})^{k \ell }.
\end{equation}
(i)  Let  $\sigma_1= \Re(s_1)$ and $\sigma_2=\Re(s_2)$.  Then 
\begin{equation} 
  \label{Bbd}
|\mathcal{B}(s_1,s_2)| \ll  \Theta(\sigma_1+\sigma_2+1,h)  \text { for } \sigma_1, \sigma_2 \ge -0.99,   \sigma_1+ \sigma_2 \ge -0.99.
\end{equation}
(ii)  We have 
\begin{equation}
    \mathcal{B}(0,0) =  \tilde{C}_{k,\ell} g_{k,\ell}(h).
\end{equation}
(iii)
\begin{equation}
  \mathcal{B}^{(i_1,i_2)}(0,0)   \ll_{k,\ell}
  \Theta(1,h) (\log \log h)^{i_1+i_2}.
\end{equation}
\end{lemma}
We also require a bound for a certain zeta integral. 
\begin{lemma} \label{zetaintegral}
Let $0<\varepsilon <1$,  $r \in \mathbb{N}$, $s\in \mathbb{C}$  with $\Re(s) \ge -\frac{1}{r}$, then 
\begin{equation}
   \int_{-\infty}^{\infty} 
   \min\Big( \frac{1}{|s|}, \frac{\varepsilon^{-1}}{|s(s+1)|} \Big) |\zeta(s+1)|^{r} dt
   \ll_{r} \varepsilon^{-1} \text{ where } s=\sigma+it.
\end{equation}
\end{lemma}
The next lemma is used to bound $\Theta(z,h)$ when $\Re(z) < 1$. 
\begin{lemma}
Let $\kappa \in [0.5,1)$.  There exists $x_{\kappa} > 0$ such that if $x \ge x_{\kappa}$, then 
\begin{equation}
  \label{primesumbd}
  \sum_{p \le x} p^{-\kappa} \le 
    \Big(\frac{12.68 \kk}{(1-\kk)^2} +3.17  \Big)
    \frac{x^{1-\kk}}{ \log x}.
\end{equation}
\end{lemma}
\begin{proof}  By Theorem 1 of \cite{RS} it follows that
\begin{equation}
 \label{Chebyshev}
\pi(x) \le \frac{3.17x}{\log x}  \text{ for } x \ge 2.
\end{equation}  
By partial summation
\begin{equation}
\begin{split}
   \sum_{2 \le p \le x} p^{-\kappa} & =  \left. \frac{\pi(t)}{t^{\kappa}}  \right|_{2}^{x}
   + \kk \int_{2}^{x} \frac{\pi(t)}{t^{\kappa+1}} dt 
   \le \frac{\pi(x)}{x^{\kappa}}+ \kk \int_{2}^{x} \frac{\pi(t)}{t^{\kappa+1}} dt  \\
   & \le 3.17 \Big( 
    \frac{x^{1-\kk}}{\log x} 
    + \kk \int_{2}^{x} \frac{1}{t^{\kk} \log t} dt 
   \Big),
\end{split}
\end{equation}
by \eqref{Chebyshev}.  We now bound the integral.  Let $y \in (2,x)$ and thus    
\begin{align*}
   \int_{2}^{x} \frac{1}{t^{\kk} \log t} dt  
   & = \int_{2}^{y} \frac{1}{t^{\kk} \log t} dt   
   +\int_{y}^{x}  \frac{1}{t^{\kk} \log t} dt   \\
   & \le \frac{y-2}{2^{\kk} \log 2} + \frac{1}{\log y} \int_{y}^{x} t^{-\kk} dt \\
   & \le \frac{y}{\sqrt{2} \log 2} + \frac{x^{1-\kk}}{(1-\kk) \log y}. 
\end{align*}
For $x$ sufficiently large, there exists $y \in (2,x)$ such that
\begin{equation}
   \label{ychoice}
   \frac{y}{\sqrt{2} \log 2} =\frac{x^{1-\kk}}{(1-\kk) \log y}.
\end{equation}
Moreover, \eqref{ychoice} implies that $\log y > \frac{1-\kk}{2} \log x$.   
Thus for $x \gg_{\kappa} 1$, we have
\[
   \sum_{2 \le p \le x} p^{-\kappa}  \le 3.17   \Big(\frac{4 \kappa}{(1-\kk)^2} +1  \Big)
    \frac{x^{1-\kk}}{ \log x} 
\] 
and we obtain \eqref{primesumbd}.
\end{proof}

With these lemmas in hand, we now establish Proposition \ref{selberg}. 
\begin{proof}[Proof of Proposition \ref{selberg}]
Without less of generality, we assume that $k \le \ell$.  Note that if $k > \ell$, then we may just swap $k$ and $\ell$.
We shall give the proof in the case $h \ge 2$.  At the end of the proof we will discuss the modifications required in the 
simpler case $h=1$. 
A standard approach would be to apply Perron's formula twice. Instead, we find it simpler 
to smooth the truncated sum. 
To simplify the evaluation of the previous sum,  we insert smoothing factors.
Let $\eta$ be positive and let  $\epsilon \in (0,1)$ be a small positive number. Let $\phi=\phi_{\eta,\epsilon}(t)$ denote a smooth, non-negative function such that 
\begin{equation}
  \phi_{\eta,\epsilon}(t) = \begin{cases}
  1 & \text{ if } t \in [0,\eta], \\
  0  & \text{ if } t \in [\eta+\epsilon,\infty). 
  \end{cases}
\end{equation}
Observe that the support of $\phi$ is contained in $[0,\eta+\epsilon]$.  We also require the derivatives to satisfy 
\begin{equation}
  \label{phijbds}
\phi_{\eta,\epsilon}^{(j)}(t) \ll \epsilon^{-j}.
\end{equation}
Later, we shall choose the parameter $\eta$ to be either $1-\epsilon$ or $1$. 
  
We shall evaluate sums of the form 
\[
  \mathscr{I}(\phi) = \sum_{\substack{a,b \in \mathbb{N} \\ (a,b) \mid h}} \frac{ \tau_{k}(a)\tau_{\ell}(b)}{[a,b]} 
  \phi \Big( \frac{a}{X} \Big)    \phi \Big( \frac{b}{X} \Big). 
\]
where $\phi(t) = \phi_{\eta,\epsilon}(t)$.   We define the Mellin transform 
\begin{equation}
  \label{Phi}
  \Phi(s) = \int_{0}^{\infty} \phi(t) t^{s-1} dt. 
\end{equation}
This is absolutely convergent for $\Re(s) >0$.  
By Mellin inversion, we have 
\begin{equation}
  \label{melinv}
  \phi(t) = \frac{1}{2 \pi i} \int_{(c)} \Phi(s) t^{-s} ds
\end{equation}
where $c >0$.   By two applications of \eqref{melinv} 
\begin{equation}
   \label{sumint}
      \mathscr{I}(\phi) 
   =  \frac{1}{(2 \pi i)^2} \int_{(c_1)} \int_{(c_2)} \mathcal{A}(s_1,s_2) X^{s_1+s_2} \Phi(s_1) \Phi(s_2) ds_1 ds_2 
\end{equation}
where $c_1,c_2 >0$, and 
\[
  \mathcal{A}(s_1,s_2) = \sum_{\substack{ a,b \ge 1 \\ (a,b) \mid h}} \frac{\tau_{k}(a) \tau_{\ell}(b)}{[a,b] a^{s_1} b^{s_2}}. 
\]
The general approach to evaluate \eqref{sumint} is to move each of the contours to the left of $\Re(s_1)=0$ and 
$\Re(s_2)=0$  and apply the residue theorem.
The integrand in \eqref{sumint}  has poles at $s_1=0$ and $s_2=0$ arising from $\mathcal{A}(s_1,s_2)$ and 
from $\Phi(s_1)$ and $\Phi(s_2)$.  A main term will arise from these poles.    The new contours will 
contribute an error term.  In order to evaluate the residue and the error terms we need  to understand 
the behaviour of $\mathcal{A}(s_1,s_2)$, $\Phi(s_1)$, and $\Phi(s_2)$ near the poles at $s_1=0$ and $s_2=0$ and we need
to provide bounds for these functions when $\Im(s_1)$ and $\Im(s_2)$ are large. 
First, we consider the behaviour of $\Phi(s)$.   By an integration by parts, 
it follows that  
\begin{equation}
  \label{Psidefinition}
 \Phi(s) = \frac{1}{s} \Psi(s) 
\end{equation}
where 
\begin{equation}
   \label{Psi2}
    \Psi(s) = -\int_{0}^{\infty}  \phi'(t) t^s dt.
\end{equation} 
This is originally valid for $\Re(s) >0$.  However, it is clear that $\Psi(s)$ is an entire function. 
Thus $\Phi(s)$ is holomorphic everywhere on $\mathbb{C}$ with the exception of a simple pole at $s=0$.  
Note that we have the Laurent expansion
\begin{equation}
  \label{Philaurexp}
   \Phi(s) = \frac{\Psi(0)}{s} + \Psi'(0) + \frac{\Psi''(0)}{2} s + \cdots. 
\end{equation}
We shall require some bounds for the expressions $\Psi^{(j)}(0)$.  
Observe that 
\begin{equation}
  \label{Psij0}
  \Psi^{(j)}(0) = - \int_{\eta}^{\eta+\epsilon} \phi'(t) (\log t)^j dt.
\end{equation}
Therefore
\begin{equation}
  \label{Psi0}
   \Psi(0) = \int_{\eta}^{\eta+\epsilon} \phi'(t) dt = \phi(\eta) =1
\end{equation}
and 
\begin{equation}
  \label{Psij0bd}
|\Psi^{(j)}(0)| \le \int_{\eta}^{\eta+\epsilon} |\phi'(t)| \max_{\eta \le t \le \eta+\epsilon} |\log t|^j dt 
  \ll \epsilon^j.
\end{equation}
Integrating \eqref{Phi} by parts $m$ times, we find that 
\[
  \Phi(s) = \frac{(-1)^m}{s(s+1) \cdots (s+m-1)} \int_{0}^{\infty} \phi^{(m)}(t) t^{s+m-1} dt,
\]
which is valid for all $s \in \mathbb{C} \setminus \{ 0 \}$. Note that for $m\ge 2$ the integrand has simple zeros at $s=-1, \ldots, -(m-1)$.    
Thus for $m \ge 1$ and $s \in \mathbb{C} \setminus \{ 0, -1, \ldots, -(m-1) \}$, 
\begin{equation}
\begin{split}
  \label{Phibd}
  |\Phi(s)| & \le \frac{1}{|s(s+1) \cdots (s+m-1)|} 
   \int_{\eta}^{\eta+\epsilon} |\phi^{(m)}(t)| t^{\sigma+m-1} dt 
   \ll_{m}  \frac{\epsilon^{1-m} (\eta+\epsilon)^{\sigma+m-1} }{|s(s+1) \cdots (s+m-1)|}. 
\end{split}
\end{equation}
Next, we  simplify the Dirichlet series $\mathcal{A}(s_1,s_2)$.  We let $g =(a,b)$ and make the variable change $a=gc$, $b=gd$ with $(c,d)=1$, 
 and group terms according to $g \mid h$
\begin{align*}
  \mathcal{A}(s_1,s_2) & = \sum_{g \mid h}  \sum_{\substack{ a,b \ge 1 \\ (a,b) =g}} \frac{\tau_{k}(a) \tau_{\ell}(b)}{[a,b] a^{s_1} b^{s_2}}
  =  \sum_{g \mid h} g \sum_{\substack{ a,b \ge 1 \\ (a,b) =g}} \frac{\tau_{k}(a) \tau_{\ell}(b)}{a^{s_1+1} b^{s_2+1}}  \\
  & = \sum_{g \mid h} \frac{1}{g^{s_1+s_2+1}} 
  \sum_{\substack{c,d \ge 1 \\ (c,d) =1}} \frac{\tau_{k}(gc) \tau_{\ell}(gd)}{c^{s_1+1} d^{s_2+1}}.
\end{align*}
The condition $(c,d)=1$ is detected by $\sum_{e \mid c, e \mid d} \mu(e)$ and thus 
\begin{equation}
   \label{A}
      \mathcal{A}(s_1,s_2)  =  \sum_{g \mid h} \frac{1}{g^{s_1+s_2+1}} 
     \sum_{e=1}^{\infty} \frac{\mu(e)}{e^{s_1+s_2+2}} 
      \sum_{\substack{ c,d \ge 1}} \frac{\tau_{k}(gec) \tau_{\ell}(ged)}{c^{s_1+1} d^{s_2+1}}.
\end{equation}
%
Inserting \eqref{sigmams} in \eqref{A}, it follows that 
\begin{equation}
  \label{Afact}
  \mathcal{A}(s_1,s_2)
   =  
  \mathcal{B}(s_1,s_2) \zeta(s_1+1)^{k} \zeta(s_2+1)^{\ell} 
\end{equation}
where 
\begin{equation}
  \label{Bv2}
   \mathcal{B}(s_1,s_2) =  \sum_{g \mid h} \frac{1}{g^{s_1+s_2+1}} 
     \sum_{e=1}^{\infty} \frac{\mu(e) \sigma_{k}(ge,s_1+1) \sigma_{\ell}(ge,s_2+1)}{e^{s_1+s_2+2}}. 
\end{equation}
By Fubini's theorem, we have
\[
 \mathscr{I}(\phi)
   = \frac{1}{(2 \pi i)^2 } \int_{(c_2)} \int_{(c_1)} \mathcal{B}(s_1,s_2) \zeta(s_1+1)^{k} \zeta(s_2+1)^{\ell}  X^{s_1+s_2} \Phi(s_1) \Phi(s_2) ds_1 ds_2. 
\]
The evaluation of multiple integrals of this type is now standard.  For instance, in \cite{GPY} and \cite{CL} more complicated integrals are treated. 
Note that the main term shall arise from the pole of order $\ell$ at $s_1=0$
and the pole of order $\ell$ at $s_2=0$ of the integrand. 
For each fixed $s_2$ with $\Re(s_2)=c_2$, the residue theorem implies that 
\begin{equation}
\begin{split}
   \label{residueformula}
   \frac{1}{2 \pi i} \int_{(c_1)} \mathcal{B}(s_1,s_2) \Phi(s_1) \zeta(s_1+1)^{k} X^{s_1} ds_1
   & = \mathrm{Res}_{s_1=0} \Big(   \mathcal{B}(s_1,s_2) \Phi(s_1) \zeta(s_1+1)^{k} X^{s_1} \Big)  
   + g(s_2)
\end{split}
\end{equation}
where
\begin{equation}
 \label{g}
g(s_2) =
   \frac{1}{2 \pi i} \int_{(c_1')} \mathcal{B}(s_1,s_2) \Phi(s_1) \zeta(s_1+1)^{k} X^{s_1} ds_1
\end{equation}
and $-1 < c_1' <0$. 
By the Laurent expansions \eqref{Philaurexp}, 
\begin{align}
    \label{Blaurent}
   \mathcal{B}(s_1,s_2) & =   \mathcal{B}(0,s_2)+  \mathcal{B}^{(1,0)}(0,s_2)s_1 + \frac{1}{2!}  \mathcal{B}^{(2,0)}(0,s_2)s_1^2
   +\cdots \\
  \label{zetallaurent}
   \zeta(s_1+1)^{k} & =  
   s_1^{-k} (\alpha_{0,k}+ \alpha_{1,k} s_1 + \alpha_{2,k} s_1^2 + \cdots ),  \text{ where } \alpha_{0,k} =1,  \\
   \label{Xslaurent}
   X^{s_1} & = 1 + (\log X) s_1 + \tfrac{1}{2} (\log X)^2 s_1^2 + \cdots 
\end{align}
it follows that 
\begin{equation}
   \label{firstresidue}
   \mathrm{Res}_{s_1=0} \Big(   \mathcal{B}(s_1,s_2) \Phi(s_1) \zeta(s_1+1)^{k} X^{s_1} \Big)
   = \sum_{\substack{i_1+i_2+i_3+i_4=k \\ i_1,i_2,i_3,i_4 \ge 0}}  
   \frac{ \mathcal{B}^{(i_1,0)}(0,s_2)\Psi^{(i_2)}(0) \alpha_{i_3, k} (\log X)^{i_4}}{i_1! i_2! i_4!}. 
\end{equation}
We now bound $g(s_2)$.  We bound $\mathcal{B}(s_1,s_2)$ using Lemma \ref{Blemma} (iii) with $i_1=i_2=0$
and we bound $\Phi(s)$ with \eqref{Philaurexp} and \eqref{Phibd} with $m=2$ to obtain
\begin{equation}
  |g(s_2)| \le \Theta(c_1'+\sigma_2+1,h) X^{c_1'} \int_{-\infty}^{\infty} \min\Big( \frac{1}{|s_1|}, \frac{\epsilon^{-1}}{|s_1(s_1+1)|} \Big) |\zeta(s_1+1)|^{k} dt_1
  \text{ where } s_1=c_1'+it_1. 
\end{equation}
It follows from Lemma \ref{zetaintegral} with $c_1' \ge -1/k$ 
\begin{equation}
  \label{gbound}
  g(s_2) \ll \epsilon^{-1} \Theta(c_1'+\sigma_2+1,h) X^{c_1'}.
\end{equation}
Thus we have
\begin{equation}
\begin{split}
   \label{Iphi}
   \mathscr{I}(\phi)
   & = \frac{1}{2 \pi i } \int_{(c_2)}   \zeta(s_2+1)^{\ell}  X^{s_2}  \Phi(s_2)
    \mathrm{Res}_{s_1=0} \Big(   \mathcal{B}(s_1,s_2) \Phi(s_1) \zeta(s_1+1)^{k} X^{s_1} \Big)
   ds_2  \\
    & + \frac{1}{2 \pi i } \int_{(c_2)}   \zeta(s_2+1)^{\ell}  X^{s_2}  \Phi(s_2) g(s_2)
   ds_2.
\end{split}
\end{equation}
By \eqref{Philaurexp}, \eqref{Phibd}, and \eqref{gbound} the second integral is bounded by 
\begin{equation}
\begin{split}
   & \epsilon^{-1} \Theta(c_1'+c_2+1,h) X^{c_1'} X^{c_2}
   \int_{-\infty}^{\infty}
   \min\Big( \frac{1}{|s_2|}, \frac{\epsilon^{-1}}{|s_2(s_2+1)|} \Big)
     dt_2 \text{ where } s_2=c_2+it_2 \\
     & \ll  \epsilon^{-2} \Theta(c_1'+c_2+1,h) X^{c_1'+c_2},
\end{split}
\end{equation}
by another application of Lemma \ref{zetaintegral}. 
Choosing $c_1'=-1/k$ and $c_2 = 0.01/k$, it follows that 
 \begin{equation}
   \mathscr{I}(\phi)
   = \frac{1}{2 \pi i } \int_{(c_2)}   \zeta(s_2+1)^{\ell}  X^{s_2}  \Phi(s_2)
    \mathrm{Res}_{s_1=0} \Big(   \mathcal{B}(s_1,s_2) \Phi(s_1) \zeta(s_1+1)^{k} X^{s_1} \Big)
   ds_2
   + O_{\ell}(  \Theta(1-\tfrac{0.99}{k},h)  \epsilon^{-2} X^{-0.99/k} ).
\end{equation}
By \eqref{firstresidue} we see that
\begin{equation}
\begin{split}
  \label{Iphiexpression}
  \mathscr{I}(\phi)
   & = 
\sum_{\substack{i_1+i_2+i_3+i_4=k \\ i_1,i_2,i_3,i_4 \ge 0}}  
   \frac{\Psi^{(i_2)}(0) \alpha_{i_3, k} (\log X)^{i_4}}{i_1! i_2! i_4!}
     \mathscr{I}_{i_1}(\phi) + O_{k}( \Theta(1-\tfrac{0.99}{k},h)  \epsilon^{-2} X^{-0.99/k}  ).
\end{split}
\end{equation}
where 
\begin{equation}
   \mathscr{I}_{i_1}(\phi)= \frac{1}{2 \pi i } \int_{(c_2)}   \mathcal{B}^{(i_1,0)}(0,s_2)   \zeta(s_2+1)^{\ell}  X^{s_2}  \Phi(s_2) ds_2
    \text{ for } i_1 \ge 0. 
\end{equation}
By an application of the residue theorem, 
\begin{equation}
   \label{Ii1phi}
     \mathscr{I}_{i_1}(\phi)
    = \mathrm{Res}_{s_2=0} \Big(  \mathcal{B}^{(i_1,0)}(0,s_2)  \Phi(s_2)   \zeta(s_2+1)^{\ell}  X^{s_2}   \Big)
     +   \frac{1}{2 \pi i } \int_{(c_2')}   \mathcal{B}^{(i_1,0)}(0,s_2)   \Phi(s_2) \zeta(s_2+1)^{\ell}  X^{s_2}  ds_2. 
\end{equation}
The second integral can be evaluated very similarly to $g(s_2)$. However,  we require a bound
for $\mathcal{B}^{(i_1,0)}(0,s_2)$ with $\sigma_2=c_2'$.  By Cauchy's integral formula   
\begin{equation}
    \mathcal{B}^{(i_1,0)}(0,s_2) = \frac{i_1 !}{2 \pi i} \int_{|z-s_2|=\delta} \frac{\mathcal{B}(0,z)}{(z-s_2)^{i_1+1}} dz
    \end{equation}
where $\delta >0$.  By an application of Lemma \ref{Blemma}, \eqref{Bbd} it follows that 
\begin{equation}
    \mathcal{B}^{(i_1,0)}(0,s_2) \ll \Theta(c_2'-\delta+1,h)  \delta^{-i_1},
\end{equation}
as long as $c_2'-\delta \ge -0.99$.  Therefore, by the above bound and Lemma \ref{zetaintegral}
\begin{equation}
\begin{split}
  \label{contourbd}
    & \frac{1}{2 \pi i } \int_{(c_2')}   \mathcal{B}^{(i_1,0)}(0,s_2)   \Phi(s_2) \zeta(s_2+1)^{\ell}  X^{s_2}  ds_2 \\
   &  \ll 
   \Theta(c_2'-\delta+1,h)  \delta^{-i_1} X^{c_2'} \int_{-\infty}^{\infty} 
   \min\Big( \frac{1}{|s_2|}, \frac{\epsilon^{-1}}{|s_2(s_2+1)|} \Big) |\zeta(s_2+1)|^{\ell} ds_2 \\
   & \ll  \epsilon^{-1} \Theta(c_2'-\delta+1,h)  \delta^{-i_1} X^{c_2'} \\
   & \ll_{\ell} \epsilon^{-1} \Theta(1-\tfrac{0.99}{\ell},h)  X^{-1/\ell},
\end{split}
\end{equation}
by the choices $c_2'=-1/\ell$ and $\delta=0.01/\ell$.  
Thus 
\begin{equation}
   \label{Ii1phi2}
   \mathscr{I}_{i_1}(\phi)
    = \mathrm{Res}_{s_2=0} \Big(  \mathcal{B}^{(i_1,0)}(0,s_2)  \Phi(s_2)   \zeta(s_2+1)^{\ell}  X^{s_2}   \Big)
    + O_{\ell} (\epsilon^{-1} \Theta(1-\tfrac{0.99}{\ell},h)  X^{-1/\ell}). 
\end{equation}
Computing the residue in \eqref{Ii1phi2} gives
\begin{align*}
    \mathscr{I}_{i_1}(\phi)= \sum_{\substack{j_1+j_2+j_3+j_4=\ell \\ j_1,j_2,j_3,j_4 \ge 0}} 
    \frac{\mathcal{B}^{(i_1,j_1)}(0,0) \Psi^{(j_2)}(0) \alpha_{j_3,\ell} (\log X)^{j_4}}{j_1! j_2! j_4!}
     + O_{\ell} (\epsilon^{-1}\Theta(1-\tfrac{0.99}{\ell},h)  X^{-1/\ell}). 
\end{align*} 
Inserting this last expression in \eqref{Iphiexpression} yields
\begin{align*}
  \label{Iphi2}
  \mathscr{I}(\phi) & = \sum_{\substack{i_1+i_2+i_3+i_4=k \\ i_1,i_2,i_3,i_4 \ge 0}}  
   \frac{ \Psi^{(i_2)}(0) \alpha_{i_3, k} (\log X)^{i_4}}{i_1! i_2! i_4!}
   \sum_{\substack{j_1+j_2+j_3+j_4=\ell \\ j_1,j_2,j_3,j_4 \ge 0}} 
    \frac{\mathcal{B}^{(i_1,j_1)}(0,0) \Psi^{(j_2)}(0) \alpha_{j_3,\ell} (\log X)^{j_4}}{j_1! j_2! j_4!} \\
    &  + O_{\ell} \Big( \Big( \sum_{i_2+i_4=k} \epsilon^{i_2}(\log X)^{i_4} \Big)\epsilon^{-1} \Theta(1-\tfrac{0.99}{\ell},h)  
    X^{-1/\ell}
    + \epsilon^{-2} \Theta(1-\tfrac{0.99}{k},h)   X^{-0.99/k}  \Big),
\end{align*}
where we have used $\Psi^{(i_2)}(0) \ll \epsilon^{i_2}$ and $\alpha_{i_3,\ell} = O_{\ell}(1)$. 
The sum in the big $O$ term is bounded by $(\log X)^{k}$ as $\epsilon < 1$. 
The main contribution to $\mathscr{I}(\phi)$ is $\mathcal{B}(0,0) (\log X)^{k+\ell}/k! \ell !$ which arises from 
$(i_1,i_2,i_3,i_4)=(0,0,0,k)$ and $(j_1,j_2,j_3,j_4)=(0,0,0,\ell)$.    By \eqref{B00bd} and \eqref{Psij0bd} 
the remaining terms are bounded by 
\begin{align*}
  \ll_{\ell} 
  \sideset{}{'}
   \sum_{\substack{i_1+i_2+i_3+i_4=k \\   j_1+j_2+j_3+j_4 = \ell} }
   \frac{\Theta(1,h) (\log \log h)^{i_1+j_1} \epsilon^{i_2+j_2}  (\log X)^{i_4+j_4}}{i_1! i_2! i_4! j_1! j_2! j_4!} 
\end{align*}
where $'$ in the summation indicates that the terms $(i_1,i_2,i_3,i_4)=(0,0,0,\ell)$ and $(j_1,j_2,j_3,j_4)=(0,0,0, \ell)$ have been excluded.    Since $\epsilon < 1$ and either $i_4 \le k-1$ or 
$j_4 \le \ell-1$, it follows that 
the remaining terms are bounded by 
\[
  \ll_{\ell} \Theta(1,h)  \sum_{\alpha+\beta \le k+ \ell -1} (\log \log h)^{\alpha} (\log X)^{\beta}
  \ll_{\ell} \Theta(1,h) (\log X)^{k+\ell-1} \log \log h.
\]
Combining the above facts, we find
\begin{equation} 
  \label{Iphiform}
  \mathscr{I}(\phi) =  \frac{\tilde{C}_{k,\ell} g_{k,\ell}(h)}{k!  \ell !}  (\log X)^{k+ \ell} + O_{\ell}( 
  \Theta(1,h) (\log X)^{k+\ell-1} \log \log h
  +  ((\log X)^{k}\epsilon^{-1}  +\epsilon^{-2}) 
   \Theta(1-\tfrac{0.99}{k},h)  X^{-0.99/\ell} 
  )
\end{equation}
since $k \le \ell$.  
We now remove the smooth weight to obtain an asymptotic formula for the truncated sum. 
Let 
\begin{equation}
 \phi_{-}(t)=\phi_{1-\epsilon,\epsilon}(t) \text{ and }
 \phi_{+}(t)=\phi_{1,\epsilon}(t)
\end{equation}
be the functions corresponding to the choices 
$\eta=1-\epsilon$ and $\eta =1$.  Note that $\phi^{-}(t)$ and $\phi^{+}(t)$
are  a smooth minorant and majorant of $\mathds{1}_{[0,1]}(t)$, the indicator function of $[0,1]$.
It follows that 
\begin{equation}
  \label{bounds}
  \mathscr{I}(\phi_{-}) \le   \sum_{\substack{a,b \le X \\ (a,b) \mid h}} \frac{ \tau_{k}(a)\tau_{\ell}(b)}{[a,b]} 
  \le  \mathscr{I}(\phi_{+}). 
\end{equation}
From \eqref{Iphiform} and \eqref{bounds} and recalling that $k=\min(k,\ell)=m$ and $\ell=\max(k,\ell)=M$ we have 
\begin{equation}
\begin{split}
  \label{sumidentity}
    \sum_{\substack{a,b \le X \\ (a,b) \mid h}} \frac{ \tau_{k}(a)\tau_{\ell}(b)}{[a,b]}   & =  \frac{\tilde{C}_{k,\ell} g_{k,\ell}(h)}{k! \ell !}  (\log X)^{k+\ell} + O_{k,\ell}( 
  \Theta(1,h) (\log X)^{k+\ell-1} \log \log h
  +  ((\log X)^{m}\epsilon^{-1}  +\epsilon^{-2}) 
   \Theta(1-\tfrac{0.99}{m},h)  X^{-0.99/M}) \\
   & = \frac{\tilde{C}_{k,\ell} g_{k,\ell}(h)}{k! \ell ! }  (\log X)^{k+ \ell} + O_{k,\ell}( 
  \Theta(1,h) (\log X)^{k+\ell-1} \log \log h
  +   \Theta(1-\tfrac{0.99}{m},h)  (\log X)^{2m}
    X^{-0.99/M}),
\end{split}
\end{equation}
by the choice $\epsilon=(\log X)^{-m}$. 
Finally, we bound $\Theta(\kappa,h)$ where $\kappa=1-\tfrac{0.99}{m}$.   We have 
\[
  \log \Theta(\kappa,h)=k \ell  \sum_{p \mid h} \log(1+p^{-\kappa})
  \le k\ell  \sum_{p \mid h} p^{-\kappa} 
\]
since $\log(1+x) \le x$ for $x \ge 0$. 
Let $\omega(h)$ denote the number of prime divisors of $h$.  
If $h \ge h_0(k,\ell)$, then 
\begin{equation}
\begin{split}
    \log \Theta(\kappa,h) & \le k \ell  \Big(  \sum_{p \le \log h} p^{-\kappa} + \frac{1}{(\log h)^{\kappa}} 
    \omega(h) \Big) \\
    & \le k \ell    \Big(
     \Big(\frac{12.68 \kk}{(1-\kk)^2} +3.17  \Big)
    \frac{(\log h)^{1-\kk}}{ \log \log h} 
  +  \frac{1.3841 (\log h)^{1-\kk}}{\log \log h} \Big)
\end{split}
\end{equation}
by \eqref{primesumbd} and  Th\'{e}or\`{e}me 11 of \cite[Robin]{R}.  It follows that
\begin{equation}
   \Theta(1-\tfrac{0.99}{m},h) \le \exp 
  \Big( 
   (12.94 m^4-12.81 m^3 +4.52 m^2 )
    \frac{(\log h)^{1-0.99/m}}{ \log \log h}  
  \Big). 
\end{equation}
Combining this with \eqref{sumidentity} completes the proof in the case $h \ge h_0(k,\ell)$.   
If $h \in [2,h_0(k,\ell)]$, it follows that  
\[
    \Theta(\kappa,h) \le  \exp \Big(  k\ell  \sum_{p \mid h} p^{-\frac{1}{2}}  \Big) \le  \exp(C_0(k,\ell)). 
\]
Inserting this in \eqref{sumidentity} establishes the proof if $h \in [2,h_0(k,\ell)]$.  

Finally, we mention the modifications in the simplest case $h=1$.   In this case, we can show that 
$\mathcal{B}(s_1,s_2)$ defined by \eqref{B} satisfies $|\mathcal{B}(s_1,s_2)| \ll 1$ for $\Re(s_1), \Re(s_2) \ge -0.99$,
$\mathcal{B}(0,0)=\tilde{C}_{\ell}$, $\mathcal{B}^{(i_1,i_2)}(0,0) \ll 1$.  Using these facts instead of Lemma \ref{B2} and following
the above argument leads to the desired result. 
\end{proof}
The proof of Proposition has been reduced to establishing Lemma \eqref{Blemma}.
\begin{proof}[Proof of Proposition \ref{Blemma}]
Throughout this proof $\sigma_1 = \Re(s_1)$ and $\sigma_2 = \Re(s_2)$.  
It will also be convenient to set $a_1,a_2 \in (0,1)$.   At the end of the proof we shall choose $a_1=a_2=0.99$. 
We begin by using Lemma \ref{f1f2prod}
with $f_1(n) = \sigma_{k}(n,s_1+1)$, $f_2(n) 
=  \sigma_{\ell}(n,s_2+1)$, 
$\tau_1 = s_1+s_2+1$, and $\tau_2 = s_1+s_2+2$  it follows from \eqref{Sproduct0} that 
\begin{equation}
\begin{split}
  \label{Bformula}
  \mathcal{B}(s_1,s_2) =
  \sum_{g \mid h} \frac{1}{g^{s_1+s_2+1}}
   &  \prod_{(p,g)=1}  \Big( 1
            - \frac{\sigma_{k}(p,s_1+1)
         \sigma_{\ell}(p,s_2+1)}{p^{s_1+s_2+2}} \Big) \\
   & \times  
          \prod_{p^{\alpha} \mid \mid g}  \Big(
          \sigma_{k}(p^{\alpha},s_1+1) \sigma_{\ell}(p^{\alpha},s_2+1)
            - \frac{\sigma_{k}(p^{\alpha+1},s_1+1)
         \sigma_{\ell}(p^{\alpha+1},s_2+1)}{p^{s_1+s_2+2}} \Big). 
\end{split}
\end{equation}
%
%
We now bound this expression.
By \eqref{sigmaps}, we note that 
\begin{equation}
   \label{Cp0}
   1   - \frac{\sigma_{k}(p,s_1+1)
         \sigma_{\ell}(p,s_2+1)}{p^{s_1+s_2+2}}
      = Q(p^{-s_1-1},p^{-s_2-1})
\end{equation}
where $
Q(x,y)=(1-x)^{k}+(1-y)^{\ell} - (1-x)^{k}(1-y)^{\ell}$.
By Taylor expansion  
\begin{equation}
\begin{split}
  Q(x,y) & = \Big( 1 -k x + \binom{k}{2} x^2   \Big)
  + \Big( 1 - \ell y + \binom{\ell}{2} y^2  \Big) 
   -  \Big( 1 -k x + \binom{k}{2} x^2  \Big)  \Big( 1 -\ell y + \binom{\ell}{2} y^2  \Big) 
  +O_{k,\ell}(|x|^3+|y|^3) \\
  & = 1 -k \ell  xy + O_{k,\ell} (|x|^2 |y| + |x| |y|^2 + |x|^2 |y|^2 + |x|^3 + |y|^3) \\
  & = 1 -k \ell  xy + O_{k,\ell} ( |x|^3+|y|^3),
\end{split}
\end{equation}
since $|x|,|y| \le 1$.   It follows that 
\begin{equation}
\begin{split}
   \label{Cp01}
    1   - \frac{\sigma_{k}(p,s_1+1)
         \sigma_{\ell}(p,s_2+1)}{p^{s_1+s_2+2}} & = 1- \frac{k \ell }{p^{s_1+s_2+2}} + O_{k,\ell}(p^{-3-3 \sigma_1}+p^{-3-3 \sigma_2}) \\
   &  = 1- \frac{k \ell }{p^{s_1+s_2+2}} + O_{k,\ell}( p^{-3+3\max(a_1,a_2)}) \text{ for } \sigma_1 \ge -a_1, \sigma_2 \ge -a_2.
\end{split}
\end{equation}
By \eqref{ngid}, we have that $
   \sigma_{k}(p^{j},s+1) =\tau_{k}(p^{j}) H_{k,j}(p^{-s-1})$.
For $j \ge 1$, we have by \eqref{ngid}
\begin{equation}
\begin{split}
  \label{Cpj2}
    &  \sigma_{k}(p^{\alpha},s_1+1) \sigma_{\ell}(p^{\alpha},s_2+1)
            - \frac{\sigma_{k}(p^{\alpha+1},s_1+1)
         \sigma_{\ell}(p^{\alpha+1},s_2+1)}{p^{s_1+s_2+2}} \\
  &  =   \tau_{k}(p^j)
   \tau_{\ell}(p^j) H_{k,j}(p^{-s_1-1}) H_{\ell,j}(p^{-s_2-1})
          - \frac{ \tau_{k}(p^{j+1}) \tau_{\ell}(p^{j+1})  H_{k,j+1}(p^{-s_1-1})H_{\ell,j+1}(p^{-s_2-1})}{p^{s_1+s_2+2}} \\
             & =    \tau_{k}(p^j) \tau_{\ell}(p^j)
           (1+O_{k,\ell}(p^{-\sigma_1-1}+p^{-\sigma_2-1}))
             - \frac{  \tau_{k}(p^{j+1}) \tau_{\ell}(p^{j+1})}{p^{s_1+s_2+2}} 
             (1+O_{k,\ell} (p^{-\sigma_1-1}+p^{-\sigma_2-1})) \\
             &  =  \tau_{k}(p^j) \tau_{\ell}(p^j)
             \Big(1+O_{k,\ell}(p^{-\sigma_1-1}+p^{-\sigma_2-1})+ O_{k,\ell} \Big( k \ell \frac{(1+\tfrac{j}{k})(1+\tfrac{j}{\ell})}{(1+j)^2} p^{-\sigma_1-\sigma_2-2}
             \Big)  \Big) \\
             & =   \tau_{k}(p^j) \tau_{\ell}(p^j)
             \Big(1+O_{k,\ell}(p^{-\sigma_1-1}+p^{-\sigma_2-1}) \Big),
\end{split}   
\end{equation}
since $\tau_{k}(p^{j+1}) = \frac{k+j}{j+1 } \tau_{k}(p^j)$.  
Using \eqref{Cp01} and the last equation, we have 
\[
   |\mathcal{B}(s_1,s_2)|  \le 
   \sum_{g \mid h} \frac{1}{g^{\sigma_1+\sigma_2+1}}
  \prod_{(p,g)=1} \Big(1+O \Big( \frac{k \ell}{p^{\sigma_1+\sigma_2+2}} + p^{-3+a_1} + p^{-3+a_2} \Big) \Big) 
  \prod_{p^{\alpha} \mid \mid g} \tau_{k}(p^j) \tau_{\ell}(p^j)
         \Big(1+O_{\ell}(p^{-1+\max(a_1,a_2)}) \Big). 
\]         
Since $\sigma_1+\sigma_2 \ge -0.99$ the first product is absolutely convergent.    
It follows that 
 \begin{equation}     
         |\mathcal{B}(s_1,s_2)|    \ll \sum_{g \mid h} \frac{\tau_{k}(g) \tau_{\ell}(g) j(g)}{g^{\sigma_1+\sigma_2+1}} 
 \end{equation}
 where $j(g) := \prod_{p \mid g} (1+Cp^{-1+\max(a_1,a_2)})$, and $C = C(k,\ell) >0$.  
 By multiplicativity, it follows that
 \begin{align*}
      |\mathcal{B}(s_1,s_2)| & \ll        \prod_{p^{\alpha} \mid \mid h} \sum_{a=0}^{\alpha} 
            \frac{\tau_{k}(p^a) \tau_{\ell}(p^a) j(p^a)}{(p^a)^{\sigma_1+\sigma_2+1}} 
             \ll \prod_{p \mid h} \Big( 1 + \frac{k \ell }{p^{\sigma_1+\sigma_2+1}} \Big) 
            = \Theta(\sigma_1+\sigma_2+1, h), 
\end{align*}
valid for $\sigma_1 \ge -a_1, \sigma_2 \ge -a_2$, and $\sigma_1 + \sigma_2 \ge -0.99$.

We now establish part $(ii)$.  By \eqref{Cp01} it follows that there exists a prime $p_0=p_0(k,\ell)$ such that if $p \ge p_0$, 
then $\mathcal{C}(p,0,s_1,s_2) \ne 0$ for $\sigma_1, \sigma_2 \ge 10^{-2}$.   Also, observe that by \eqref{Cp0}
$\mathcal{C}(p,0,0,0) = (1-p^{-1})^{k}+(1-p^{-1})^{\ell} - (1-p^{-1})^{k+ \ell} \ne 0$.  Since for each $2 \le p \le p_0$, $\mathcal{C}(p,0,s_1,s_2)$
is a continuous function of $s_1$  and $s_2$, there exists $\varepsilon_{0} \in (0,10^{-2})$ such that $\mathcal{C}(p,0,s_1,s_2) \ne 0$ 
for $p \in [2,p_0]$ and $|s_1| < \varepsilon_0$ and $|s_2| < \varepsilon_0$. Combining these facts, it follows that for all primes $p$
and $s_1,s_2$ satisfying  $|s_1| < \varepsilon_0$ and $|s_2| < \varepsilon_0$, that $\mathcal{C}(p,0,s_1,s_2) \ne 0$.
Thus we may apply Lemma \ref{f1f2prod} (ii).  Let 
\begin{equation}
\label{Cpj}
     \mathcal{C}(p,j,s_1,s_2)  =    \frac{\sigma_{k}(p^j,s_1+1) \sigma_{\ell}(p^j,s_2+1)}
            {p^{j(s_1+s_2+1)}} - \frac{\sigma_{k}(p^{j+1},s_1+1)
         \sigma_{\ell}(p^{j+1},s_2+1)}{p^{(j+1)(s_1+s_2+1)+1}}.
\end{equation}
Since $\mathcal{C}(p,0,s_1,s_2) \ne 0$ for $|s_1|,|s_2| \le \varepsilon_0 <10^{-2}$, Lemma  \eqref{f1f2prod} (ii)
implies that
\begin{equation}
\begin{split} 
    \label{Bfact}
     \mathcal{B}(s_1,s_2)     =  \mathcal{B}_1(s_1,s_2) \mathcal{B}_2(s_1,s_2) \text{ for } |s_1|,|s_2| \le \varepsilon_0,
\end{split}
\end{equation}
where 
\begin{equation}
  \label{B1}
   \mathcal{B}_1(s_1,s_2)  =   \prod_{p}   \mathcal{C}(p,0,s_1,s_2)
\end{equation}
and
\begin{equation}
   \label{B2}
    \mathcal{B}_2(s_1,s_2)  = \prod_{p^{\alpha} \mid \mid h} \sum_{j=0}^{\alpha} 
       \mathcal{C}(p,j,s_1,s_2)   \mathcal{C}(p,0,s_1,s_2)^{-1}.
\end{equation}
We first determine the value of $\mathcal{B}(0,0)$. 
It follows from \eqref{B1}, \eqref{Cp0}, and \eqref{tilCell} that $\mathcal{B}_1(0,0) = \tilde{C}_{k,\ell}$.
Similarly, it follows from  and \eqref{B2}, \eqref{Cpj}, \eqref{Cp0}, and  \eqref{gldef} that $\mathcal{B}_2(0,0) = g_{k,\ell}(h)$.
Hence, $\mathcal{B}(0,0)=\tilde{C}_{k,\ell}g_{k,\ell}(h)$.  

We now establish $(iii)$.  Let $\varepsilon_0$ be as in part (ii).  It shall be convenient to define 
\[
    D_1 = \{ s_1 \in \mathbb{C} \ | \ |s_1| < \varepsilon_0 \} \text{ and }
     D_2 = \{ s_2 \in \mathbb{C} \ | \ |s_2| < \varepsilon_0 \}.
\]
First observe that by the definition \eqref{Cpj} and \eqref{Cp01} and \eqref{Cpj2} we have  
\begin{align}
  \label{Cp02ndbd}
      \mathcal{C}(p,0,s_1,s_2)    & = 1- \frac{k\ell}{p^{s_1+s_2+2}} 
      + O_{k,\ell}( p^{-2.7}) 
      \text{ for } \sigma_1, \sigma_2 \ge -10^{-1}, \\
    \label{Cpj2ndbd}
      \mathcal{C}(p,j,s_1,s_2)   & =  \frac{\tau_{k}(p^j) \tau_{\ell}(p^j)}{p^{j(s_1+s_2+1)}  }
             \Big(1+O_{k,\ell}(p^{-0.9}) \Big)  
              \text{ for } \sigma_1, \sigma_2 \ge -10^{-1}.
\end{align}
In addition, for every $\varepsilon >0$,  $\tau_{k}(p^j), \tau_{\ell}(p^j) \ll_{\ell} p^{j\varepsilon/2}$ and we also have the estimate 
\begin{equation}
  \label{Cpj3}
   \mathcal{C}(p,j,s_1,s_2) \ll p^{-j(\sigma_1+\sigma_2+1- \varepsilon)}
   \text{ for } \sigma_1 \ge -a_1, \sigma_2 \ge -a_2.
\end{equation}
 From \eqref{B1} and \eqref{Cp02ndbd} we see that
\begin{equation}
\begin{split}
\mathcal{B}_1(s_1,s_2) & =   \prod_{p} \Big( 1-\frac{k \ell}{p^{s_1+s_2+2}} + O_{k,\ell}(p^{-2.7}) \Big ) 
 \ll_{k,\ell} 1
\text{ for } \sigma_1, \sigma_2 \ge -10^{-1}.
       \end{split}
\end{equation}
By two applications of Cauchy's integral formula,  
\begin{equation}
   \label{B1bd}
    \mathcal{B}_1^{(i,j)}(s_1,s_2) = O_{k,\ell}(1)  \text{ for } \sigma_1, \sigma_2 \ge -10^{-2}
\end{equation}
where $ \mathcal{B}_1^{(i,j)}$ is defined by \eqref{partials}.
We now estimate $\mathcal{B}_2(s_1,s_2)$. 
First, we examine each local factor at $p$ of $\mathcal{B}_2$. 
Since $D_1 \times D_2 \subset \{ s_1 \in \mathbb{C} \ | \ \sigma_1 \ge -10^{-1} \} 
\times \{ s_2 \in \mathbb{C} \ | \ \sigma_2 \ge -10^{-1} \}$ it follows from \eqref{Cp02ndbd},
\eqref{Cpj2ndbd}, and \eqref{Cpj3} that
\begin{equation} 
\begin{split}
   \label{B2pestimate}
    \sum_{j=0}^{\alpha} 
       \frac{\mathcal{C}(p,j,s_1,s_2)}{   \mathcal{C}(p,0,s_1,s_2)}
   & = 1 + \frac{\frac{k \ell }{p^{s_1+s_2+1}} + O(p^{-1.5})  }{1+ O(p^{-1.8})  }
  + \frac{\sum_{j=2}^{\alpha} p^{-0.7j}  }{1+ O(p^{-1.8})  } \\
   & = 1 + \frac{k \ell}{p^{s_1+s_2+1}} + O(p^{-1.4}) \text{ for } (s_1,s_2) \in D_1 \times D_2. 
\end{split}
\end{equation}
Hence we can factor out a term $(1+ \frac{1}{p^{s_1+s_2+1}})^{k \ell}$ from \eqref{B2}.  Therefore
we may write 
\begin{equation}
  \label{B2fact}
  \mathcal{B}_2(s_1,s_2) = \Theta(s_1+s_1+1,h) \mathcal{B}_3(s_1,s_2) \text{ for } (s_1,s_2) \in D_1 \times D_2,
\end{equation}
where we recall that $\Theta(z,h) = \prod_{p \mid h} (1+ p^{-z})^{k \ell}$ and
\begin{equation}
  \label{B3}
  \mathcal{B}_3(s_1,s_2) =   \prod_{p^{\alpha} \mid \mid h} \Big( \sum_{j=0}^{\alpha} 
       \mathcal{C}(p,j,s_1,s_2)   \mathcal{C}(p,0,s_1,s_2)^{-1}   \Big) \Big(1 + \frac{1}{p^{s_1+s_2+1}} \Big)^{-k \ell}.
\end{equation}
It follows from \eqref{B2pestimate} and \eqref{B3} that 
\[
   \mathcal{B}_3(s_1,s_2) =   \prod_{p   \mid h} (1+O_{k,\ell}(p^{-\sigma_1-\sigma_2 -2}))
   \ll_{k,\ell} \prod_{p} (1 + O_{k,\ell}(p^{-1.8}))
    \ll_{k,\ell} 1 \text{ for } (s_1,s_2) \in D_1 \times D_2.  
\]
By Cauchy's integral formula it follows that 
\begin{equation}
   \label{B3bd}
   \mathcal{B}_{3}^{(i_1,i_2)} (0,0) \ll_{k,\ell} 1.
\end{equation}
We also require an estimate for the partial derivatives of $\mathcal{B}$.  By \eqref{Bfact} and \eqref{B2fact} 
it follows that 
\begin{equation}
  \label{Bfact2}
\mathcal{B}(s_1,s_2) = \Theta(s_1+s_1+1,h)   \tilde{B}(s_1,s_2)
   \text{ for } (s_1,s_2) \in D_1 \times D_2,
\end{equation}
where $\tilde{\mathcal{B}}(s_1,s_2) =\mathcal{B}_1(s_1,s_2) \mathcal{B}_3(s_1,s_2)$.  Note that the generalized product
rule, \eqref{B1bd}, and \eqref{B3bd} imply
\begin{equation}
  \label{tilBbd}
  \tilde{B}^{(i_1,i_2)}(0,0) \ll_{k,\ell} 1. 
\end{equation}
By two applications of the generalized product rule to \eqref{Bfact2} 
\[
  \mathcal{B}^{(i_1,i_2)} (s_1,s_2) = \sum_{a_1+a_2=i_1} \binom{i_1}{a_1}
  \sum_{a_3+a_4=i_2} \binom{i_2}{a_3} 
 \Big(  \frac{ \partial^{a_1} }{ \partial s_{1}^{a_1}  }
  \frac{ \partial^{a_3}}{\partial s_{2}^{a_3}}   \Theta(s_1+s_1+1,h) \Big) \tilde{\mathcal{B}}^{(a_2,a_4)} (s_1,s_2). 
\]
Note that 
\begin{equation}
  \label{thetaid}
  \frac{ \partial^{a_1} }{ \partial s_{1}^{a_1}  }
  \frac{ \partial^{a_3}}{\partial s_{2}^{a_3}}   \Theta(s_1+s_1+1,h)
  = \frac{d^{a_1+a_3}}{dz^{a_1+a_3}} \Theta(z,h) \Big|_{z=s_1+s_2+1}. 
\end{equation}
By \eqref{B3bd} and \eqref{thetaid} it follows that
\[
  \mathcal{B}^{(i_1,i_2)} (0,0) \ll_{k,\ell} \sum_{\alpha=0}^{i_1+i_2} \Theta^{(\alpha)}(1,h)
\]
We now demonstrate for $\alpha \ge 1$
\begin{equation}
  \label{thetaalphabd}
   \Theta^{(\alpha)}(1,h) \ll \Theta(1,h) (\log \log h)^{\alpha}.
\end{equation}
We begin by remarking that 
\begin{equation}
  \label{thetaprime}
\Theta^{(1)}(z,h) = -k \ell  \Theta(z,h) \eta(z,h)
\end{equation}
where 
\begin{equation}
  \label{eta}
\eta(z,h)= \sum_{p \mid h} \frac{\log p}{p^z+1}.
\end{equation} 
By the product rule it follows that 
\begin{equation}
   \label{thetaalpharecur}
    \Theta^{(\alpha)}(z,h) = - k \ell  \sum_{u_1+u_2 =\alpha-1}  \binom{\alpha-1}{u_1} \Theta^{(u_1)}(z,h) \eta^{(u_2)}(z,h).
\end{equation}
A calculation demonstrates that for $u \ge 0$
\[
  \eta^{(u)}(1,h) =
  \sum_{p \mid h} \frac{(\log p)^u}{p} \sum_{j=1}^{\infty} \frac{(-1)^{j-1}(-j)^u}{p^{j-1}}
\]
and thus 
\begin{equation}
  \label{etaubd}
  \eta^{(u)}(1,h) \ll \sum_{p \mid h}   \frac{(\log p)^u}{p}
  \ll \sum_{p \le \log h}  \frac{(\log p)^u}{p} +  \frac{(\log \log h)^u}{\log h} \sum_{p \mid h} 1
  \ll (\log \log h)^u. 
\end{equation}
We now show \eqref{thetaalphabd}.  The case $\alpha=1$ follows from \eqref{thetaprime}
and \eqref{etaubd} with $u =1$.  By induction, using \eqref{thetaalpharecur}
and \eqref{etaubd}, we establish  \eqref{thetaalphabd} for all $\alpha \ge 1$. 
From \eqref{thetaalphabd} we now have
\begin{equation}
  \label{B00bd}
  \mathcal{B}^{(i_1,i_2)} (0,0) \ll_{k,\ell}
  \Theta(1,h) (\log \log h)^{i_1+i_2}.
\end{equation}
\end{proof}
\begin{proof}[Proof of Proposition \ref{zetaintegral}]
Using the first bound, we find the contribution from $|t| \le 1$ to the integral is $O(1)$.  
We now treat the range $|t| \ge 1$.  It is convenient to set  $I(\tau, t) = \int_{0}^{t} |\zeta(\tau+iu)|^{r} du$.
It is well known that for every $\varepsilon >0$,
\begin{equation}
   \label{mvt}
  I(\tau,t) \ll t^{1+\varepsilon} \text{ for } \tau \ge 1-1/r. 
\end{equation}
This follows from \cite[Theorems 7.5,7.7]{Tit}. Note that in the case of Theorem 7.7 of \cite{Tit}, 
the bound $ I(\tau,t) \ll t$ for $\tau > 1-1/r$ is stated, however a minor modification of the proof yields
\eqref{mvt}. 
Since the integrand is even with respect to $t$ the remaining range is 
\begin{equation}
\begin{split}
  2 \varepsilon^{-1} \int_{1}^{\infty} |s(s+1)|^{-1}  |\zeta(s+1)|^{r} dt
  & \ll \varepsilon^{-1}  \int_{1}^{\infty} |\zeta(\sigma+1+it)|^{r} t^{-2} dt \\
  & \ll  \varepsilon^{-1} \Big( -I(\sigma+1,1) +2 \int_{1}^{\infty} I(\sigma+1,t) t^{-3} dt \Big) \\
  & \ll \varepsilon^{-1}\Big( 1 + \int_{1}^{\infty}  t^{-2+\varepsilon} dt \Big)
  \ll \varepsilon^{-1}
\end{split} 
\end{equation}
by an integration by parts and \eqref{mvt}. 
\end{proof}
\section{A probabilistic method for determining main term of $D_{k,\ell}(x,h)$}  \label{Probmethod}

In this section, we use a simple heuristic probabilistic method to rederive  the conjectured formula
\begin{equation}
  \label{conjform}
  D_{k,\ell}(x,h) \sim \frac{c_{k,\ell}(h)}{(k-1)! (\ell-1)!} x (\log x)^{k+\ell-2}
  \text{ for } 1 \le h \le x^{1-\varepsilon}
\end{equation}
for $x$ large, $\varepsilon$ arbitrarily small, and recall that
$C_{k,\ell}$ is defined by \eqref{Ckl} and $f_{k,\ell}(h)$ is defined by \eqref{fklpa}. 
In section three, we derived this conjecture using the $\delta$-method.

The argument in this section has been used to derive conjectures for $\sum_{n \le x} \Lambda(n)\Lambda(n+h)$, 
where $\Lambda(n)$ is the von Mangoldt function (for full details see \cite{Ch} and \cite{Po}).   
The extension to the case of multiplicative functions was explained to the first author by Andrew Granville.
We now proceed with our heuristic derivation of \eqref{conjform}. It is well known that 
\begin{equation}
   \label{Dkx}
   \sum_{n \le x} \tau_k(n) \sim \frac{1}{(k-1)!}  x(\log x)^{k-1}.
\end{equation}
It follows that on average $\tau_k(n)$ in the interval $[1,x]$ is  $\tfrac{1}{(k-1)!}  x(\log x)^{k-1}$.
Similarly, for $1 \le h \le x^{1-\varepsilon}$, $\tau_{\ell}(n+h)$ in the interval $[1,x]$ is also $\tfrac{1}{(\ell-1)!}  x(\log x)^{\ell-1}$.
Thus it is reasonable to believe that for $1 \le h \le x^{1-\varepsilon}$,  $\tau_{k}(n)\tau_{\ell}(n+h)$ is on average $\tfrac{1}{(k-1)!(\ell-1)!}(\log x)^{k+\ell-2}$
in $[1,x]$.   
However, we must take into consideration that the values of $\tau_k(n)$ and $\tau_{\ell}(n+h)$
are not independent.   For instance, if $h=p$ is prime, then if $p \nmid n$ we also have $p \nmid n+h$.  
The factor $C_{k,\ell}f_{k,\ell}(h)$ in \eqref{conjform} accounts for such local considerations. 
In order to make this precise we define a sequence of random variables $(X_p)_{p \text{ prime}}$  by 
\[
   X_{p}(n) = \tau_{k}(p^{\text{ord}_p(n)}) 
\]
where $\text{ord}_p(\cdot)$ is the $p$-adic valuation.   Furthermore, we define 
\[
   Y_{p}(n) =
   \tau_{\ell}(p^{\text{ord}_p(n+h)}).
\]
Associated to a random variable $Y: \mathbb{N} \to \mathbb{C}$ with image $\text{im}(Y) = \{ Y(n) \ | \ n\in \mathbb{N} \}$, its expected value to be
\begin{equation}
  \label{ev}
  \mathbb{E}(Y) = \sum_{i \in \text{im}(Y)}  i \cdot \mathbb{P}(Y= i)
\end{equation}
where for $B \subseteq \mathbb{N}$, 
\begin{equation}
  \label{P}
  \mathbb{P}(B) = \lim_{X \to \infty} \frac{\# \{  1 \le n \le X \ | \ n \in B \} }{X}. 
\end{equation}
With these definitions in hand, it is natural to make the following conjecture. \\
{\bf Conjecture}.  For $\varepsilon \in (0,1)$, $x$ large, and  $1 \le h \le x^{1-\varepsilon}$,
\begin{equation}
   \label{Dkxhprobconj}
   \frac{1}{x} D_{k,\ell}(x,h)  \sim \Bigg( \prod_{p}  \frac{\mathbb{E}(X_p Y_p )}{\mathbb{E}(X_p) \mathbb{E}(Y_p)} \Bigg ) 
  \Big(  \frac{1}{x} \sum_{n \le x} \tau_k(n)  \Big)
  \Big(  \frac{1}{x} \sum_{n \le x} \tau_{\ell}(n+h)  \Big)
\end{equation}
as $x \to \infty$. 

The product in the above conjecture is the correction factor taking into account that the values of $\tau_k(n)$
and $\tau_{\ell}(n+h)$ are not independent. 
Each local factor in the product measures the lack of independence of $X_p$ and $Y_p$.
We shall prove that the product equals $c_{k,\ell}(h)=C_{k,\ell} f_{k,\ell}(h)$, which we computed earlier via the $\delta$-method.   
\begin{proposition} \label{product}
Let $k, \ell, h \in \mathbb{N}$.  Then
\begin{equation}
     \label{cklhexpectedvalue}
     c_{k,\ell}(h) =\prod_{p}  \frac{\mathbb{E}(X_p Y_p )}{\mathbb{E}(X_p) \mathbb{E}(Y_p)}. 
\end{equation}
\end{proposition}
By \eqref{Dkxhprobconj}, Proposition \ref{product}, and \eqref{Dkx} we have that 
\begin{equation}
\begin{split}
   \frac{1}{x} D_{k,\ell}(x,h) \sim \frac{c_{k,\ell}(h)}{(k-1)!(\ell-1)!} \log^{k-1}(x) \log^{\ell-1}(x+h)
 \sim   \frac{c_{k,\ell}(h)}{(k-1)!(\ell-1)!}  \log^{k+\ell-2}(x)  
\end{split}
\end{equation}
as $x \to \infty$ for  $h \le  x^{1-\varepsilon}$.
This yields the Additive Divisor Conjecture
(simplified version) stated in the introduction.  

The above proposition is deduced from the next lemma.
\begin{lemma} \label{evalue}
(i)  For every prime $p$, 
\begin{equation}
     \mathbb{E}(X_{p}) 
         = \Big( 1-\frac{1}{p} \Big)^{-(k-1)}
    \text{ and }
      \mathbb{E}(Y_{p}) 
    = \Big( 1-\frac{1}{p} \Big)^{-(\ell-1)}
\end{equation}
(ii) 
If $p \nmid h$, then 
\begin{equation}
    \mathbb{E}(X_{p}Y_{p}) 
  =   \Big( 1-\frac{1}{p} \Big)^{-(k-1)}+ \Big( 1-\frac{1}{p} \Big)^{-(\ell-1)}-1.
\end{equation}
(iii) If $p^{\alpha} \mid \mid h$, then 
\begin{equation}
\begin{split}
  \label{EXpXph}
 \mathbb{E}(X_{p} Y_p)  
=1 & + \sum_{i=1}^{\alpha} (\tau_k(p^i)\tau_{\ell}(p^i)-\tau_k(p^{i-1})\tau_{\ell}(p^{i-1}))X^i \\
   & +      \sum_{i=\alpha+1}^{\infty}  (\tau_k(p^{\alpha}) \tau_{\ell-1}(p^i)
   + \tau_{\ell}(p^{\alpha})   \tau_{k-1}(p^i))X^i.
\end{split}
\end{equation}
\end{lemma}
In this section we also show that Tao's probabilistic argument \cite{Tao} gives the same answer. 
\begin{proposition}  \label{Taoequality}
Let $k,\ell, h \in \mathbb{N}$.  Then 
\begin{equation}
  c_{k,\ell}(h) = \prod_{p} \mathfrak{S}_{k,\ell,h}(p) 
\end{equation}
where $\mathfrak{S}_{k,\ell,h}(p)$ is defined by \eqref{Sklhp}.  
\end{proposition}

We now demonstrate the proof of Proposition \ref{product} based on this lemma.  
\begin{proof}[Proof of Proposition \ref{product}]
If $p \nmid h$, then by Lemma \ref{evalue} (i) and (ii)
\[
    \frac{\mathbb{E}(X_p Y_p )}{\mathbb{E}(X_p) \mathbb{E}(Y_p)} = 
    \frac{\Big( 1-\frac{1}{p} \Big)^{-(k-1)}+\Big( 1-\frac{1}{p} \Big)^{-(\ell-1)}-1}{  \Big( 1-\frac{1}{p} \Big)^{-(k-1)} \Big( 1-\frac{1}{p} \Big)^{-(\ell-1)}}
    = \Big( 1-\frac{1}{p} \Big)^{k-1} + \Big( 1-\frac{1}{p} \Big)^{\ell-1} - \Big( 1-\frac{1}{p} \Big)^{k+\ell-2}.
\]
Therefore
\begin{equation}
\begin{split}
 \label{fkpa3}
 & \prod_{p}  \frac{\mathbb{E}(X_p Y_p )}{\mathbb{E}(X_p) \mathbb{E}(Y_p)} 
   = C_{k,\ell} \prod_{p^{\alpha} \mid \mid h}  \frac{ \mathbb{E}(X_p Y_p )}{\mathbb{E}(X_p) \mathbb{E}(Y_p)} \\
  & = C_{k,\ell} \prod_{p^{\alpha} \mid \mid h}  \frac{
 1 + \sum_{i=1}^{\alpha} (\tau_k(p^i)\tau_{\ell}(p^i)-\tau_k(p^{i-1})\tau_{\ell}(p^{i-1}))X^i 
    +     \sum_{i=\alpha+1}^{\infty} ( \tau_k(p^{\alpha})  \tau_{\ell-1}(p^i)
   + \tau_{\ell}(p^{\alpha})   \tau_{k-1}(p^i))X^i
 }{\Big( \Big( 1-\frac{1}{p} \Big)^{-(k-1)}+\Big( 1-\frac{1}{p} \Big)^{-(\ell-1)}-1 \Big)} \\
 & =  C_{k,\ell} f_{k,\ell}(h)
\end{split}
\end{equation}
by an application of Lemma \ref{evalue} part (iii). 
\end{proof}

Before establishing Lemma \ref{evalue}, we make a few observations. 
\begin{equation}
  \label{pndiv}
   \mathbb{P} \Big( \{ n \in \mathbb{N} \ | \ p \nmid n \} \Big) = 1-\frac{1}{p}
\end{equation}
and for $i \ge 1$, 
\begin{equation}
   \label{piddn}
     \mathbb{P} \Big( \{ n \in \mathbb{N} \ | \ p^i \mid \mid n\} \Big) = \frac{1}{p^i}-\frac{1}{p^{i+1}},
\end{equation}
where $\mathbb{P}$ is defined by \eqref{P}. 
Idenitity \eqref{pndiv} is since  $n$ lies in $p-1$ of $p$ residue classes modulo $p$ and \eqref{piddn} 
follows from writing $n=p^i n'$ where $(n',p)=1$.  
\begin{proof}[Proof of Lemma \ref{evalue}] (i)
First, we compute $\mathbb{E}(X_p)$.  The values of $X_p$ are precisely $\tau_k(p^i)$ for $i \ge 0$. 
Note that if $i=0$, then $X_p=1$.  This means that $p \nmid n$ and the probability of this occurring is 
$1-\frac{1}{p}$, since $n$ lies in $p-1$ of $p$ residue classes modulo $p$.  Now $X_p(n) = \tau_k(p^i)$ with $i \ge 1$, 
precisely when $p^i \mid \mid n$.  This occurs with probability $\frac{1}{p^i} -\frac{1}{p^{i+1}}$.   Therefore
\[
   \mathbb{E}(X_p) = 1-\frac{1}{p} + \sum_{i=0}^{\infty} \tau_k(p^i) \Big(\frac{1}{p^i} -\frac{1}{p^{i+1}} \Big)
   = \Big( 1-\frac{1}{p} \Big) \sum_{i=0}^{\infty} \frac{\tau_k(p^i)}{ p^{i}}
    = \Big( 1-\frac{1}{p} \Big)^{-(k-1)}. 
\]
A similar argument establishes 
$\mathbb{E}(Y_p) 
    = ( 1-\frac{1}{p} )^{-(\ell-1)}$. \\
(ii) We now compute $\mathbb{E}(X_p Y_p )$, in the case $p \nmid h$. 
If $n  \not \equiv 0,-h (\text{mod } p)$, then $X_p(n)=Y_p(n)=1$.  
The probability of this occurring is $\frac{p-2}{p}=1-\frac{2}{p}$. 
If $n  \equiv 0 (\text{mod } p)$ and $p^{i} \mid \mid n$ with $i \ge 1$, then $p \nmid n+h$. 
Therefore $X_{p}(n)= \tau_k(p^i)$ and $Y_p(n)=1$.   The probability of this 
occurring is $\frac{1}{p^i} - \frac{1}{p^{i+1}}$.   Similarly, if  $n  \equiv -h (\text{mod } p)$ and $p^{i} \mid \mid n+h$ with $i \ge 1$, 
then $p \nmid n$. 
Therefore $X_{p}(n)= 1$ and $Y_p(n)=\tau_{\ell}(p^i)$ and the probability of this 
occurring is $\frac{1}{p^i} - \frac{1}{p^{i+1}}$.  It follows that 
\[
  \mathbb{E}(X_p Y_p) = 1-\frac{2}{p} +  \sum_{i=1}^{\infty} \tau_k(p^i) \Big( \frac{1}{p^i} - \frac{1}{p^{i+1}} \Big)
  +   \sum_{i=1}^{\infty} \tau_{\ell}(p^i) \Big( \frac{1}{p^i} - \frac{1}{p^{i+1}} \Big).
\]
Now 
\begin{align*}
   \sum_{i=1}^{\infty} \tau_k(p^i) \Big( \frac{1}{p^i} - \frac{1}{p^{i+1}} \Big) 
   & = \Big( 1-\frac{1}{p} \Big)  \sum_{i=1}^{\infty} \frac{\tau_k(p^i)}{p^i} 
   = \Big( 1-\frac{1}{p} \Big) \Big( \Big(1 -\frac{1}{p} \Big)^{-k} -1 \Big)\\
  &  = \Big( 1-\frac{1}{p} \Big)^{-(k-1)}  -1 + \frac{1}{p}
\end{align*}
and thus 
\begin{align*}
\mathbb{E}(X_p Y_p) & = 
1-\frac{2}{p} + \Big( 1-\frac{1}{p} \Big)^{-(k-1)}  -1 + \frac{1}{p} +\Big( 1-\frac{1}{p} \Big)^{-(\ell-1)}  -1 + \frac{1}{p} \\
& =  \Big( 1-\frac{1}{p} \Big)^{-(k-1)} + \Big( 1-\frac{1}{p} \Big)^{-(\ell-1)} -1.
\end{align*}
(iii)We now compute $\mathbb{E}(X_p Y_p )$,  in the case $p^{\alpha} \mid \mid h$. \\
If $n \not \equiv 0 (\text{mod } p)$, then $n+h \not \equiv 0 (\text{mod } p)$.  This is since if $p \mid n+h$, then $p \mid n$ 
as $p^{\alpha} \mid \mid h$.  This case occurs with probability $1-\frac{1}{p}$ and for these $n$, $X_p(n) = Y_p(n) =1$. 
These terms contribute
\begin{equation}
   \label{cont1}
  1 \cdot \Big(1-\frac{1}{p} \Big) = 1-\frac{1}{p}
\end{equation}
to $\mathbb{E}(X_p Y_p)$.  
Now consider $p^i \mid \mid n$ with $i \ge 1$.  In this case, $X_p(n) = \tau_k(p^{i})$.  We now determine the order 
of $p$ dividing $n+h$.  Writing $n=p^i n'$ and $h=p^{\alpha} h'$ with $(n',p)=(h',p)=1$, we have 
\[
   n+h = p^i n' + p^{\alpha} h' = p^{\min(i,\alpha)} (n' p^{i- \min(i,\alpha)} +h' p^{\alpha- \min(i,\alpha)}). 
\]
Note that if $i \ne \alpha$, then $\text{ord}_{p}(n+h) = \min(i,\alpha)$ and  $Y_p(n) = \tau_{\ell}(p^{\min(i, \alpha)})$.
These terms make a contribution 
\begin{equation}
   \label{cont2}
   \sum_{\substack{i=1 \\ i \ne \alpha}}^{\infty} \tau_k(p^i) \tau_{\ell}(p^{\min(i,\alpha)}) \Big( \frac{1}{p^i} - \frac{1}{p^{i+1}} \Big). 
\end{equation}
to $\mathbb{E}(X_p Y_p)$.  
However, if $i =\alpha$, then $X_{p}(n) = \tau_k(p^{\alpha})$.  Now we determine the power of $p$ dividing
$n+h$.  Since $n+h = p^{\alpha} (n'+h')$, the $p$-adic valuation depends on the order of $p$ dividing $n'+h'$.  
Since $(n',p)=1$, it falls in $p-1$ residue classes modulo $p$.
If $n'   \not \equiv -h' (\mbox{mod }  p )$, then $\text{ord}_p(n+h) =\alpha$. 
If $n'   \equiv -h' (\mbox{mod }  p )$, then  there exists $j \ge 1$ such that $p^j \mid \mid n'+h'$
and $\text{ord}_p(n+h)=\alpha+j$.  By these observations we  have the disjoint union
\[
  \{ n \in \mathbb{N} \ | \ p^{\alpha} \mid \mid n\}
  =    A_0  \cup \bigcup_{j=1}^{\infty} A_j
\]
where $A_0 =  \{ n \in \mathbb{N} \ | \ p^{\alpha} \mid \mid n,  n'   \not \equiv -h' (\mbox{mod }  p ) \}$, $A_j=  \{ n \in \mathbb{N} \ | \ p^{\alpha} \mid \mid n,  p^j \mid \mid n'  +h'   \} $, and $n'=\tfrac{n}{p^{\alpha}}$. 
Since $(n',p)=1$ and $n' \not \equiv -h' (\mbox{mod }  p )$, it follows that $\mathbb{P}(A_0) = \frac{1}{p^{\alpha}}(1-\frac{2}{p})$
as $n'$ lies  in $p-2$ residue classes modulo $p$.  A similar calculation establishes that $\mathbb{P}(A_j)= \frac{1}{p^{\alpha}}(\frac{1}{p^j}
-\frac{1}{p^{j+1}})$.   If $n \in A_0$, then $Y_p(n) = \tau_{\ell}(p^{\alpha})$ and if $n \in A_j$, then 
$Y_p(n) =\tau_{\ell}(p^{\alpha+j})$.
The contribution from all terms with $p^{\alpha} \mid \mid n$ is 
\begin{equation}
    \label{cont3}
     \frac{ \tau_k(p^{\alpha}) \tau_{\ell}(p^{\alpha}) }{p^{\alpha}}  \Big(1 - \frac{2}{p} \Big) +
    \sum_{j=1}^{\infty} \frac{ \tau_k(p^{\alpha}) \tau_{\ell}(p^{\alpha+j})}{p^{\alpha}}  
    \Big( \frac{1}{p^{j}} - \frac{1}{p^{j+1}} \Big).
\end{equation} 
Combining \eqref{cont1}, \eqref{cont2}, and \eqref{cont3} yields
\begin{align*}
  \mathbb{E}(X_p Y_p) & = 1-\frac{1}{p} + 
  \Big ( 1 - \frac{1}{p} \Big )    \sum_{\substack{ i \ge 1 \\i \ne \alpha}} \frac{\tau_k(p^i) \tau_{\ell}(p^{\min(i,\alpha)})}{p^i}
 +  
   \frac{\tau_k(p^{\alpha})\tau_{\ell}(p^{\alpha}) }{p^{\alpha}}  \Big(1 - \frac{2}{p} \Big)  
  +
    \sum_{j=1}^{\infty} \frac{\tau_k(p^{\alpha}) \tau_{\ell}(p^{\alpha+j})}{p^{\alpha}}  
    \Big( \frac{1}{p^{j}} - \frac{1}{p^{j+1}} \Big) \\
& =  \Big ( 1 - \frac{1}{p} \Big )  \Big( 1  + \sum_{\substack{ i \ge 1 \\i \ne \alpha}} \frac{\tau_k(p^i)   \tau_{\ell}(p^{\min(i,\alpha)}) }{p^i}+ 
   \frac{\tau_k(p^{\alpha})\tau_{\ell}(p^{\alpha}) }{p^{\alpha}}  \Big(1 - \frac{2}{p} \Big) \Big(1 - \frac{1}{p} \Big)^{-1}+
\sum_{j=1}^\infty
\frac{\tau_k(p^\alpha)\tau_{\ell}(p^{\alpha+j})}{p^{\alpha+j}}
\Big) \\
  & =  \Big ( 1 - \frac{1}{p} \Big )  \Big( 
\sum_{i=0}^{\alpha-1} \frac{\tau_k(p^i)\tau_{\ell}(p^i) }{p^i} + \sum_{j=0}^{\infty} \frac{\tau_k(p^{\alpha}) \tau_{\ell}(p^{\alpha+j})}{p^{\alpha+j}}  + \sum_{j=0}^{\infty} \frac{\tau_k(p^{\alpha+j}) \tau_{\ell}(p^{\alpha})}{p^{\alpha+j}} 
-
   \frac{\tau_k(p^{\alpha})\tau_{\ell}(p^{\alpha}) }{p^{\alpha}} \Big(1 - \frac{1}{p} \Big)^{-1} 
  \Big). 
\end{align*}
We simplify this a bit further. Setting $X=\frac{1}{p}$ we have
\begin{align*}
    \mathbb{E}(X_p Y_p) & = \sum_{i=0}^{\alpha-1} \tau_k(p^i)\tau_{\ell}(p^i)X^i -X  \sum_{i=0}^{\alpha-1} \tau_k(p^i)\tau_{\ell}(p^i) X^i
    +   \sum_{j=0}^{\infty} \tau_k(p^{\alpha}) \tau_{\ell}(p^{\alpha+j})X^{\alpha+j} 
    - X \sum_{j=0}^{\infty} \tau_k(p^{\alpha}) \tau_{\ell}(p^{\alpha+j})X^{\alpha+j}   \\
    &  +\sum_{j=0}^{\infty} \tau_k(p^{\alpha+j}) \tau_{\ell}(p^{\alpha}) X^{\alpha+j} 
    - X \sum_{j=0}^{\infty} \tau_k(p^{\alpha+j}) \tau_{\ell}(p^{\alpha}) X^{\alpha+j}  - \tau_k(p^{\alpha})\tau_{\ell}(p^{\alpha}) X^{\alpha} \\
    & = 1 + \sum_{i=1}^{\alpha-1} (\tau_k(p^i)\tau_{\ell}(p^i)-\tau_k(p^{i-1})\tau_{\ell}(p^{i-1}))X^i -\tau_k(p^{\al-1})\tau_{\ell}(p^{\al-1}) X^{\al} \\
   &  + \tau_k(p^{\alpha}) \Big( \tau_{\ell}(p^{\alpha}) X^{\alpha} + \sum_{i=\alpha+1}^{\infty} (\tau_{\ell}(p^i)-\tau_{\ell}(p^{i-1}))X^i \Big)  \\
  &   + \tau_{\ell}(p^{\alpha}) \Big( \tau_{k}(p^{\alpha}) X^{\alpha} + \sum_{i=\alpha+1}^{\infty} (\tau_{k}(p^i)-\tau_{k}(p^{i-1}))X^i \Big) 
   - \tau_k(p^{\alpha})\tau_{\ell}(p^{\alpha}) X^{\alpha} \\
   & =   1 + \sum_{i=1}^{\alpha} (\tau_k(p^i)\tau_{\ell}(p^i)-\tau_k(p^{i-1})\tau_{\ell}(p^{i-1}))X^i
   +  \tau_k(p^{\alpha})   \sum_{i=\alpha+1}^{\infty} \tau_{\ell-1}(p^i)X^i  
   + \tau_{\ell}(p^{\alpha})   \sum_{i=\alpha+1}^{\infty} \tau_{k-1}(p^i)X^i,
\end{align*}
by two applications of \eqref{recurrence}.
This establishes \eqref{EXpXph} and completes the proof of Lemma \ref{evalue}.
\end{proof}
Finally, we establish Proposition \ref{Taoequality}. 
\begin{proof}[Proof of Proposition \ref{Taoequality}]
First we show that expressions given for $P_{k,\ell,p}(j)$ in \eqref{Pklpj} and \eqref{Pklpj2} are equal.
Observe that 
\begin{align*}
   \sum_{k'=2}^{k} \binom{k-k'+j-1}{k-k'}
   =\sum_{i=0}^{k-2} \binom{i+j-1}{i} 
   =\sum_{i=0}^{k-2} \tau_{i}(p^j) = \tau_{k-1}(p^j)
\end{align*}
and 
\begin{align*}
   \sum_{k'=2}^{k} \binom{k-k'+j-1}{k-k'}  \Big( \frac{p}{p-1} \Big)^{k'-1}
   = \sum_{i=0}^{k-2} \binom{i+j-1}{i}   \Big( \frac{p}{p-1} \Big)^{k-i-1}.
\end{align*}
Therefore 
\begin{align*}
   P_{k,\ell,p}(j)  &=  \tau_{k-1}(p^j)  \sum_{i=0}^{\ell-2} \binom{i+j-1}{i}   \Big( \frac{p}{p-1} \Big)^{\ell-i-1} 
    + \tau_{\ell-1}(p^j)  \sum_{i=0}^{k-2} \binom{i+j-1}{i}   \Big( \frac{p}{p-1} \Big)^{k-i-1} 
   -\tau_{k-1}(p^j) \tau_{\ell-1}(p^j).
\end{align*}
From \eqref{cklNgThom}, \eqref{cklTao}, and \eqref{Pklpj2}
it suffices to prove 
$\mathcal{L}_{k,\ell}(\alpha) = \tilde{\mathcal{R}}_{k,\ell}(\alpha)$
where 
\begin{align*}
   \mathcal{L}_{k,\ell}(\alpha) 
  & =1+\sum_{i=1}^{\alpha} ( \tau_k(p^i) \tau_{\ell}(p^i)- \tau_k(p^{i-1}) \tau_{\ell}(p^{i-1}))X^{i}   
         + \tau_k(p^{\alpha})  \sum_{i=\alpha+1}^{\infty} \tau_{\ell-1}(p^i) X^i
     + \tau_{\ell}(p^{\alpha})  \sum_{i=\alpha+1}^{\infty} \tau_{k-1}(p^i) X^i, \\
  \tilde{\mathcal{R}}_{k,\ell}(\alpha) & = \sum_{j=0}^{\alpha} X^j \Big(  \tau_{k-1}(p^j)  \sum_{i=0}^{\ell-2} \binom{i+j-1}{i}   \Big( \frac{p}{p-1} \Big)^{\ell-i-1} \\
   & + \tau_{\ell-1}(p^j)  \sum_{i=0}^{k-2} \binom{i+j-1}{i}   \Big( \frac{p}{p-1} \Big)^{k-i-1}
    -\tau_{k-1}(p^j)\tau_{\ell-1}(p^j) \Big)
 \end{align*}
 and $X=\frac{1}{p}$. 
 We shall prove this by induction.  As before, we have 
$\mathcal{L}_{k,\ell}(1) = \tilde{\mathcal{R}}_{k,\ell}(1)$ where the value is given by \eqref{L1R1}.
 Now assume that for $\alpha \in \mathbb{N}$, $\mathcal{L}_{k,\ell}(\alpha) = \tilde{\mathcal{R}}_{k,\ell}(\alpha)$.
 We aim to show that $\mathcal{L}_{k,\ell}(\alpha+1) = \tilde{\mathcal{R}}_{k,\ell}(\alpha+1)$.
 Recall that we showed \eqref{Lkdifference}
\begin{align*}
  \mathcal{L}_k(\alpha+1) - \mathcal{L}_k(\alpha)  & =    \tau_k(p^{\alpha+1}) \tau_{\ell}(p^{\alpha+1})X^{\alpha+1}  +   \sum_{i=\alpha+2}^{\infty} (\tau_{k-1}(p^{\alpha+1}) \tau_{\ell-1}(p^i) 
     +   \tau_{\ell-1}(p^{\alpha+1})  \tau_{k-1}(p^i)) X^i.
\end{align*}
On the other hand
\begin{align*}
 & \tilde{\mathcal{R}}_{k,\ell}(\alpha+1)-\tilde{\mathcal{R}}_{k,\ell}(\alpha)   =  \\
  & X^{\alpha+1} \Big(  \tau_{k-1}(p^{\alpha+1})
   \sum_{i=0}^{\ell-2} \binom{i+\alpha}{i}(1-X)^{-\ell+i+1} 
   +\tau_{\ell-1}(p^{\alpha+1})
   \sum_{i=0}^{k-2} \binom{i+\alpha}{i}(1-X)^{-k+i+1} 
  -\tau_{k-1}(p^{\alpha+1}) \tau_{\ell-1}(p^{\alpha+1}) 
  \Big).
\end{align*}
We see that $\mathcal{L}_{k,\ell}(\alpha+1)  -\mathcal{L}_{k,\ell}(\alpha) = \tilde{\mathcal{R}}_{k,\ell}(\alpha+1)-
\tilde{\mathcal{R}}_{k,\ell}(\alpha)$ if and only if 
\begin{align*}
  & \tau_{k-1}(p^{\alpha+1}) \tau_{\ell-1}(p^{\alpha+1})X^{\alpha+1} 
    +  \tau_{k-1}(p^{\alpha+1}) \sum_{i=\alpha+2}^{\infty} \tau_{\ell-1}(p^i) X^i 
     +   \tau_{\ell-1}(p^{\alpha+1}) \sum_{i=\alpha+2}^{\infty} \tau_{k-1}(p^i) X^i \\
 & = X^{\alpha+1} \Big(  \tau_{k-1}(p^{\alpha+1})
   \sum_{i=0}^{\ell-2} \binom{i+\alpha}{i}(1-X)^{-\ell+i+1} 
   +\tau_{\ell-1}(p^{\alpha+1})
   \sum_{i=0}^{k-2} \binom{i+\alpha}{i}(1-X)^{-k+i+1} 
  -\tau_{k-1}(p^{\alpha+1}) \tau_{\ell-1}(p^{\alpha+1}) 
  \Big). 
\end{align*}
Rearranging this becomes
\begin{equation}
\begin{split}
  \label{reducedidentity}
  &  \tau_{k-1}(p^{\alpha+1}) \sum_{i=\alpha+1}^{\infty} \tau_{\ell-1}(p^i) X^i 
     +   \tau_{\ell-1}(p^{\alpha+1}) \sum_{i=\alpha+1}^{\infty} \tau_{k-1}(p^i) X^i \\
 & = X^{\alpha+1} \Big(  \tau_{k-1}(p^{\alpha+1})
   \sum_{i=0}^{\ell-2} \binom{i+\alpha}{i}(1-X)^{-\ell+i+1} 
   +\tau_{\ell-1}(p^{\alpha+1})
   \sum_{i=0}^{k-2} \binom{i+\alpha}{i}(1-X)^{-k+i+1} 
  \Big). 
\end{split}
\end{equation}
Observe that the left hand side of \eqref{reducedidentity} is
\begin{equation}
  \label{leftside}
  \sum_{i=\alpha+1}^{\infty} ( 
   \tau_{k-1}(p^{\alpha+1}) \tau_{\ell-1}(p^i)
     +   \tau_{\ell-1}(p^{\alpha+1})\tau_{k-1}(p^i)) X^i.
\end{equation}
Using \eqref{divseries}
we see that the right hand side of 
\eqref{reducedidentity} is 
\begin{equation}
\begin{split}
  \label{rightside}
  & X^{\alpha+1} \Big(  \tau_{k-1}(p^{\alpha+1})
   \sum_{u=0}^{\ell-2} \binom{u+\alpha}{u} \sum_{u=0}^{\infty} \tau_{\ell-u-1}(p^j) X^j 
   +\tau_{\ell-1}(p^{\alpha+1})
   \sum_{u=0}^{k-2} \binom{u+\alpha}{u}\sum_{u=0}^{\infty} \tau_{k-u-1}(p^j) X^j  \Big) \\
   & = \sum_{i=\alpha+1}^{\infty}
   \Big(   \tau_{k-1}(p^{\alpha+1})  \sum_{u=0}^{\ell-2} \binom{u+\alpha}{u}\tau_{\ell-u-1}(p^{i-(\alpha+1)})
   +  \tau_{\ell-1}(p^{\alpha+1})  \sum_{u=0}^{k-2} \binom{u+\alpha}{u}\tau_{k-u-1}(p^{i-(\alpha+1)})
   \Big)
   X^{i}.
\end{split}
\end{equation}
Therefore we see that \eqref{reducedidentity} holds if the  coefficient  of $X^i$ in \eqref{leftside} and \eqref{rightside} are equal.
In fact, we shall show that for $i \ge \alpha+1$ that 
\begin{equation}
  \label{finalidentity}
    \tau_{\ell-1}(p^i) = \sum_{u=0}^{\ell-2} \binom{u+\alpha}{u}\tau_{\ell-u-1}(p^{i-(\alpha+1)})
    \text{ and }
    \tau_{k-1}(p^i) =\sum_{u=0}^{k-2} \binom{u+\alpha}{u}\tau_{k-u-1}(p^{i-(\alpha+1)}).
\end{equation}
Observe that the second identity is the same as the first with $k$ and $\ell$ swapped.  
 Thus it suffices to establish the first identity in \eqref{finalidentity}.  By \eqref{taukpj} this reads as 
 \begin{equation}
     \binom{\ell+i-2}{i} 
     = \sum_{u=0}^{\ell-2} \binom{u+\alpha}{u}
     \binom{\ell-u-1+i-(\al+1)-1 }{i-(\al+1)} 
     \text{ for } i \ge \al+1. 
 \end{equation}
 Letting $L=\ell-1$ and $j=i-(\al+1)$, this is the same as 
 \begin{equation}
     \binom{L+j+\al}{1+j+\al} 
     = \sum_{u=0}^{L-1} \binom{u+\alpha}{u}
     \binom{L-u+j-1 }{j} 
     \text{ for }  j \ge 0. 
 \end{equation} 
 However, this is identity (1.78) of \cite{G} 
\begin{equation}
  \binom{a+r+n+1}{n} = \sum_{u=0}^{n} \binom{a+u}{u} \binom{r+n-u}{n-u}
\end{equation}
with $n=L-1$, $a=\al$, and $r=j$.   It follows that $\mathcal{L}_{k,\ell}(\alpha+1) - \mathcal{L}_{k,\ell}(\alpha)
= \tilde{\mathcal{R}}_{k,\ell}(\alpha+1)- \tilde{\mathcal{R}}_{k,\ell}(\alpha)$
and thus $ \mathcal{L}_{k,\ell}(\alpha)
= \tilde{\mathcal{R}}_{k,\ell}(\alpha)$ for all $\alpha \in \mathbb{N}$. 
\end{proof}

\section{Concluding remarks} \label{Conclusion}
In this article, we studied the sum  $D_{k,\ell}(x,h)$.
Lower bounds for this sum were obtained and the main term in its conjectured asymptotic was studied.  
We now mention several avenues of possible future research.  \\
\begin{enumerate}
\item Improve the lower bounds for $D_{k,\ell}(x,h)$.   One might attempt to use inequalities of the shape 
\[
   \tau_k(n) \ge \sum_{ \substack{\alpha_1 \cdots \alpha_k = n   \\ \prod_{i \in \mathcal{J}}  \alpha_i \le x^{\beta_i}}}  1
\]
where  $\mathcal{J}$ ranges over certain subsets of $\{ 1, \ldots, k \}$ and $\sum_{i=1}^{k} \beta_i \le 1/2$. 
\item Establish a version of the uniform bound \eqref{Danbd2}, making the $k$ dependence explicit.  
Currently, even the bound \eqref{Danbd2} for $h \le x^{C}$ does not appear in a published reference. 
\item It seems possible that the probabilistic method of section \ref{Probmethod} can be used to obtain the 
full main term asymptotic for $D_{k,\ell}(x,h)$.  
\item Study the more general sums
\begin{equation}
  \label{generaldivisor}
  \sum_{n \le x} \tau_{k_1}(n+h_1) \cdots \tau_{k_r}(n+h_r)
\end{equation}
where $r \ge 2$, $k_1, \ldots, k_r \in \mathbb{R}^{+}$ and $h_1, \ldots, h_r \in \mathbb{Z}$.  
It is likely that the methods of this article may be applied to obtain 
lower bounds for \eqref{generaldivisor} of the correct order of magnitude and to write down conjectural asymptotic formula for this sum. 
The asymptotic evaluation of \eqref{generaldivisor} is an open problem and this is well-known to the experts. 
\footnote{The first author gave a talk  at ICERM (Brown University) on Nov. 12, 2015  and mentioned this open problem.}
For instance, it is an open problem to evaluate the sum 
\[
  \sum_{n \le x} \tau(n) \tau(n+1) \tau(n+2).  
\]
It should be noted that Blomer \cite{Bl} recently succeeded in evaluating the triple correlation sum 
\[
     \sum_{x \le n \le 2x} \tau(n-h) \tau(n) \tau(n+h)  
\]
on average over $h$.  
\item  Study the lower order terms in the main term asymptotic for $D_{k,\ell}(x,h)$.
More precisely, determine explicit expressions for the coefficients $\alpha_i(h)$ for $0 \le i \le 2k-3$
and numerically study the size of $|D_{k,\ell}(x,h)-P_{2k-2;h}(\log x)|$ with $h$ as a function of $x$.  
This might provide evidence towards the true sizes of the constants $\theta_k$ and $\beta_k$
in Conjecture \eqref{adddivconj}.  Furthermore, it seems possible to use the probabilistic method of 
section 5 to calculate the lower order terms.
\end{enumerate}

\noindent {\bf Acknowlegements} \\
We thank Kevin Henriot for discussions concerning his work \cite{He} and for informing us of the 
unpublished work of Daniel \cite{Da} and for providing us with a sketch of a proof of \eqref{Danbd2}.
We also thank Professor Andrew Granville for explaining  the probabilistic argument applied in section 4.
Thank-you to Professor Terry  Tao for communications regarding this article and his blogpost \cite{Tao}. 
The first author is supported by an NSERC Discovery Grant and the second author was supported with an NSERC USRA award for this research. 


\end{document}